\documentclass{scrarticle}

\usepackage[colorlinks=true,linkcolor=teal,citecolor=teal,hypertexnames=false]{hyperref}
\usepackage[margin=2.5cm]{geometry}
\usepackage{amsmath,amssymb,amsthm}
\usepackage{enumitem}
\usepackage{comment}
\usepackage{microtype}
\usepackage[english]{babel}
\usepackage[utf8]{inputenc}
\usepackage{mathrsfs}
\usepackage{bbm}

\usepackage{tikz}
\usetikzlibrary{matrix,arrows}
\usepackage{tikz-cd}
\usetikzlibrary{cd}
\usepackage{comment}

\mathcode`A="7041 \mathcode`B="7042 \mathcode`C="7043 \mathcode`D="7044
\mathcode`E="7045 \mathcode`F="7046 \mathcode`G="7047 \mathcode`H="7048
\mathcode`I="7049 \mathcode`J="704A \mathcode`K="704B \mathcode`L="704C
\mathcode`M="704D \mathcode`N="704E \mathcode`O="704F \mathcode`P="7050
\mathcode`Q="7051 \mathcode`R="7052 \mathcode`S="7053 \mathcode`T="7054
\mathcode`U="7055 \mathcode`V="7056 \mathcode`W="7057 \mathcode`X="7058
\mathcode`Y="7059 \mathcode`Z="705A

\usepackage{calrsfs}
\DeclareMathAlphabet{\pazocal}{OMS}{zplm}{m}{n}

\newcommand\bF{\mathbb F}
\newcommand\bZ{\mathbb Z}
\newcommand\bN{\mathbb N}

\newcommand\C{\mathbb C}
\newcommand\ZZ{\mathbb Z}
\newcommand\FF{\mathbb F}
\newcommand\Fq{\FF_q}
\newcommand\Spec{\mathrm{Spec}~}
\newcommand\coeff{\mathrm{coeff}}

\newcommand{\T} {\mathbb{T}}
\newcommand{\CI}{\C_{\infty}}

\newcommand{\sf}{\mathfrak{sf}}  

\newcommand{\eps}{\varepsilon}

\newcommand\cO{\mathcal O}
\newcommand\cG{\mathcal G}

\newcommand\Frac{\operatorname{Frac}}

\newcommand\End{\operatorname{End}}
\newcommand\id{\operatorname{id}}
\newcommand\Frob{\operatorname{Frob}}
\newcommand\Hom{\mathrm{Hom}}
\newcommand\co{\mathrm{co}}

\newcommand\Mat{\operatorname{Mat}}
\newcommand{\Restau}{\operatorname{Res}_{\tau}}
\newcommand{\tRestau}{\widetilde{\operatorname{Res}}_{\tau}}
\newcommand\Ga{\mathbb{G}_a}

\newcommand{\carl}{\mathsf{C}}

\newcommand{\KI}{K_\infty}

\newcommand{\ka}{\Omega_{A/\mathbb{F}}}

\newcommand{\ps}[1]{[\![#1]\!]}  
\newcommand{\ls}[1]{(\!(#1)\!)}

\renewcommand{\sp}[1]{[#1]} 
\newcommand{\sps}[1]{[\![#1]\!]} 
\newcommand{\perf}{\mathrm{perf}}
\newcommand{\dk}{\check{\kappa}} 
\newcommand{\de}{\check{e}} 

\newtheorem{Theorem}{Theorem}[section]
\newtheorem{Lemma}[Theorem]{Lemma}

\newtheorem{Conjecture}[Theorem]{Conjecture}
\newtheorem{Proposition}[Theorem]{Proposition}
\newtheorem{Corollary}[Theorem]{Corollary}
\theoremstyle{definition}
\newtheorem{Definition}[Theorem]{Definition}

\newtheorem*{Definition*}{Definition}
\newtheorem{Remark}[Theorem]{Remark}
\newtheorem*{rem*}{Remark}

\makeatletter
\newcommand{\psum}{%
  \DOTSB\mathop{%
    \smash{\sideset{}{'}\sum}%
    \vphantom{\sum}%
  }\slimits@
}
\makeatother

\title{Pairing Anderson motives via formal residues in the Frobenius endomorphism}
\date{}
\author{Quentin Gazda\footnote{Sorbonne Universit\'e and Universit\'e Paris Cit\'e, CNRS, IMJ-PRG, 75005 Paris, \emph{Email:} \texttt{quentin@gazda.fr}}, Andreas Maurischat\footnote{RWTH Aachen University, 52072 Aachen, \emph{Email:} \texttt{maurischat@combi.rwth-aachen.de} \\ The second author had received support from the SFB-TRR 195 “Symbolic Tools in Mathematics and their Application” of the German Research Foundation (DFG).}}

\begin{document}

\maketitle

\setcounter{tocdepth}{1}

\begin{abstract}
{\small
Anderson modules form a generalization of Drinfeld modules and are commonly understood as the counterpart of abelian varieties but with function field coefficients. In an attempt to study their ``motivic theory'', two objects of semilinear algebra are attached to an Anderson module: its \emph{motive} and its \emph{dual motive}. While the former is better suited to follow the analogy with Grothendieck motives, the latter has proven much useful in the study of transcendence questions in positive characteristic.\\
Despite sharing similar definitions, the relationship between motives and dual motives has remained nebulous. Over perfect fields, it was only proved recently by the second author that the finite generation of the motive is equivalent to the finite generation of the dual motive, answering a long-standing open question in function field arithmetic (the ``abelian equals $A$-finite'' theorem). \\
This work constructs a perfect pairing among the motive and the dual motive of an Anderson module, with values in a module of differentials, thus answering a question raised by Hartl and Juschka. Our construction involves taking the residue of certain formal power series in the Frobenius endomorphism. Although it may seem peculiar, this pairing is natural and compatible with certain base change. It also comes with several new consequences in function field arithmetic; for example, we generalize the ``abelian equals $A$-finite'' theorem to all perfect algebras. 
}
\end{abstract}

\tableofcontents
\vspace{1cm}

\section{Introduction}
\subsection{Context}
Let $\mathbb{F}$ be a finite field with $q$ elements. Let $C$ be a geometrically irreducible smooth projective curve over $\mathbb{F}$ and $\infty$ a closed point on $C$. We consider the $\mathbb{F}$-algebra $A$ of functions on $C$ that are regular away from $\infty$.\\
Unlabeled tensor products will always be over $\bF$, and unlabeled $\Hom$-sets will always be homomorphisms of $\bF$-vector spaces or $\bF$-algebras.

Generalizing the pioneering work of Drinfeld, Anderson introduced certain $A$-module schemes which serve as analogues of abelian varieties in function fields arithmetic, but with $A$ as the coefficient ring instead of $\mathbb{Z}$. To an Anderson $A$-module $E$ over an $A$-algebra base $R$, Anderson attaches two objects from semilinear algebra: primarily, its \emph{$A$-motive} $M(E)$ which corresponds to the $A\otimes R$-module of homomorphisms from $E$ to $\mathbb{G}_a$ as $\mathbb{F}$-vector spaces schemes; it acquires a left action of the Frobenius $\tau$ of $\mathbb{G}_a$. In unpublished work reproduced in \cite{abp}, Anderson also attaches the \emph{dual $A$-motive} $N(E)$ which rather consists in homomorphisms from $\mathbb{G}_a$ to $E$. Similarly, $N(E)$ acquires a right action of $\tau$. We refer to Section~\ref{sec:objects-from-FF-arithmetic} for details. \\

Despite their similar definitions, the relation between $M(E)$ and $N(E)$ as modules over the ring $A\otimes R$ is quite subtle. When $R$ is a perfect field, it was only proved recently by the second author that the finite generation of the former amounts to that of the latter~\cite{maurischat}. Under the additional assumption 
that $R=\mathbb{C}_{\infty}$\footnote{We remind the reader unfamiliar with function field arithmetic notations that $\mathbb{C}_{\infty}$ denotes the completion of an algebraic closure of $\Frac(A)$ at the place $\infty$.}
 is a complete algebraically closed $A$-algebra, Hartl and Juschka \cite{hartl-juschka} further showed the 
existence of a perfect pairing among $\tau^* M(E)$ and $N(E)$ with values in the module of differentials of $A\otimes_{\mathbb{F}} \CI$ 
over $\CI$, thereby establishing the isomorphism class of $M(E)$ in terms of that of $N(E)$. In \emph{loc.\,cit.}, 
the authors asked whether it is possible to give an explicit definition of this pairing (\emph{cf}. \cite[Question 2.5.15]{hartl-juschka}). \\

One aim of this text is to answer Hartl and Juschka's question in giving a canonical construction of the pairing they introduced. 
We provide such an answer for perfect $A$-algebras $R$, and not only $R=\CI$.

\subsection{Main construction and results}
Let $R$ be an $A$-algebra with structure morphism $\iota$. To fully appreciate our main contribution, we recall the definition of Anderson modules as generalized by Hartl in \cite{hartl}.
\begin{Definition}\label{def:anderson-modules}
An \emph{Anderson $A$-module of dimension $d$ over $R$} is a smooth affine $A$-module scheme $E$ over $R$ having the following properties:
\begin{enumerate}
\item there is a faithfully flat ring homomorphism $R\to S$ for which the base change $E\times_R S$ is isomorphic to the $d$th power of the additive $\mathbb{F}$-vector space scheme over~$S$;
\item\label{item:Lie-compatible} For any $a\in A$, $\mathrm{Lie}_E (a)-\iota(a)$ seen as an endomorphism of the $R$-module $\mathrm{Lie}_E(R)$ is nilpotent. 
\end{enumerate}
\end{Definition}

Drinfeld modules are instances of Anderson $A$-modules of dimension one.\\

Given an Anderson $A$-module $E$ over $R$, we may consider the following two groups of $\mathbb{F}$-vector space scheme homomorphisms over $R$:
\[
M(E):=\Hom_{\mathbb{F}}(E,\mathbb{G}_a) \quad \text{and} \quad N(E):=\Hom_{\mathbb{F}}(\mathbb{G}_a,E).
\] 
Both are naturally $A\otimes R$-modules where $A$ acts on $E$ and $R$ acts on $\mathbb{G}_a$. The former is usually referred to as \emph{the motive of $E$} and the latter as \emph{the dual motive of $E$}. In order to avoid this confusing terminology, we rather use the naming \emph{comotive of $E$} for $N(E)$. We shall say that $E$ is \emph{abelian} if $M(E)$ is finite projective\footnote{If we use the terminology ``finite'' for modules, we always mean ``finitely generated''.} over $A\otimes R$; respectively, we say that $E$ is \emph{$A$-finite} (or, for consistency, \emph{coabelian}) if $N(E)$ is finite projective.\\

We denote by $\ka$ the module of K\"ahler differentials of $A$ over $\mathbb{F}$. Our main result is the following theorem.
\begin{Theorem}\label{thm:main}
Let $R$ be a perfect $A$-algebra and let $E$ be an abelian (respectively coabelian) Anderson $A$-module over $R$. There is a natural $A\otimes R$-linear perfect pairing
\begin{equation}\label{eq:intro-pairing}
\tau^*M(E)\otimes_{A\otimes R} N(E)\longrightarrow \ka\otimes_{\mathbb{F}} R
\end{equation}
which is compatible with base change. In the case where $R=\mathbb{C}_{\infty}$, the induced isomorphism $N(E)\to \Hom(\tau^*M(E),\ka\otimes_{\mathbb{F}} R)$  is the inverse of Hartl and Juschka's $\Xi$-map (\emph{cf}. \cite[Theorem~2.5.13]{hartl-juschka}).
\end{Theorem}

In answering Hartl and Juschka's question, we give an explicit construction of the map \eqref{eq:intro-pairing}. 

The rather surprising feature of the map \eqref{eq:intro-pairing} is that it can be interpreted as taking the ``residue'' of certain formal Laurent series in the Frobenius operator of $\mathbb{G}_a$. We refer to Section \ref{sec:residue-in-tau} for the detailed construction. \\

There are several practical applications of Theorem \ref{thm:main}; we compile some of them in Section~\ref{sec:applications}. For instance, we are able to generalize the second author's main theorem in \cite{maurischat} to a large class of $A$-algebras (see Corollary \ref{cor:abelian=coabelian}), offering at once an alternative proof.
\begin{Corollary}[Abelian equals coabelian]
Assume that $R$ is perfect. Then $M(E)$ is finite projective (of constant rank) over $A\otimes R$ if, and only if $N(E)$ is.
\end{Corollary}

In Section \ref{sec:noncommut}, we study the non-commutative rings $R[\tau]$, $R[\tau,\tau^{-1}]$, $R\ps{\sigma}$ and $R\ls{\sigma}$, where $R$ is a perfect ring of characteristic $p$. The latter two play a key role in our considerations. This allows us to define topological modules $M\ls{\sigma}$ and $N\ls{\sigma}$, where $M$ and $N$ are the motive and the comotive of $E$; the necessary background on Anderson's \emph{zoo of objects} is recalled in Section \ref{sec:objects-from-FF-arithmetic}. In particular, we are able to show that the $\sigma$-adic topology on these modules coincide with the $\infty$-adic topology. Section \ref{sec:residue-in-tau} is devoted to the construction of the pairing \emph{residue-in-$\tau$} which hinges on a formal \emph{residue in $\tau$} map $R\ls{\sigma}\to R$ extracting the coefficient of $\tau^{-1}$. Using the bridge between these topologies, we are able to show the perfectness of the pairing. 
Then, in Section \ref{sec:applications} we present some applications of our main result, such as ``abelian=$A$\nobreakdash-finite'' statements, the equivalence between tensor of motives and dual motives up to isogeny, special functions, Barsotti--Weil formulas, and we also generalize the construction of twisted Anderson modules \emph{à la mode de Caen}. Last but not least, we compute some instances of the residue-in-$\tau$ pairing for Drinfeld modules, tensor powers of the Carlitz module and the so-called ``Maurischat example''. \\

\textbf{Acknowledgments.} The first author is deeply indebted to Andres Fernandez\nobreakdash-Herrero and Laurent Moret-Bailly for enlightening discussions on forms of powers of $\mathbb{G}_a$ and the ill-behaviour of the comotive in the presence of imperfect rings. Some results stemming from these discussions have been incorporated into Subsection~\ref{subsec:forms-of-Gad}. The second author wishes to thank Urs Hartl for suggesting a potential application of our main result to the theory of Anderson generating functions, which is now presented in Subsection~\ref{subsec:anderson-generating-functions}. Finally, both authors are grateful to the referee for their many helpful comments.

\section{Noncommutative rings of functions of $\tau$}\label{sec:noncommut}
Let $\FF$ be the finite field with $q$ elements, $q$ a power of a prime $p$. 
Let $R$ be a commutative $\FF$-algebra. We let $R[\tau]$ be the skew polynomial ring over $R$ in a formal variable $\tau$ subject to the relation $\tau a=a^q \tau$ for all $a\in R$. This means, the elements of $R[\tau]$ are given by finite sums $a_0+a_1\tau+\ldots+a_r\tau^r$ for some $r\geq 0$ and coefficients $a_i\in R$, addition is coefficient wise, and multiplication is bilinear satisfying $\tau a=a^q \tau$, i.e.~
\[ \left( \sum_{i} a_i\tau^i\right)\cdot \left( \sum_{j} b_j\tau^j\right) = \sum_{k} \left( \sum_{i+j=k} a_ib_j^{q^i} \right) \tau^{k}.\]

The ring $R[\tau]$ plays a prominent role in the theory of $A$-motives because it identifies with the ring of endomorphisms of $\mathbb{G}_{a,R}$ as an $\mathbb{F}$-vector space scheme; seen as a functor,
\[
\mathbb{G}_{a,R}~:~\{R\text{-algebras}\}\longrightarrow \{\mathbb{F}\text{-vector~spaces}\},\quad S \longmapsto (S,+).
\]
It is well-known that the map $R[\tau]\to \End_{\mathbb{F}-\text{vs}/R}(\mathbb{G}_{a,R})$, $\tau\mapsto \Frob_{\mathbb{G}_a}$
is a ring isomorphism, where $\Frob_{\mathbb{G}_a}$ is the $q$-Frobenius on $\mathbb{G}_a$ (\emph{cf}. \cite[Lemma 3.2]{hartl}). 

\begin{Remark}\label{rem:tensor-left-right}
Due to its non-commutativity, one has to be quite careful when manipulating the ring $R[\tau]$. For instance, although $S\otimes_R R[\tau]\cong S[\tau]$, where $R$ acts on the left of $R[\tau]$, the ring $S[\tau]$ is generally not isomorphic to $R[\tau]\otimes_R S$ if now $R$ acts on the right of $R[\tau]$.
\end{Remark}

\subsection{Colimit perfection}
The ring $R[\tau]$ is much more well-behaved\footnote{If $R=k$ is an imperfect field then $k[\tau]$ is not noetherian as a right-module over itself. To see this, let $a\in k\setminus k^p$ and consider the increasing sequence $(I_s)_{s>0}$ of right-ideals whose $s$th term is $I_s=a\tau k[\tau]+\tau a\tau k[\tau]+\ldots+\tau^{s-1}a\tau k[\tau]$. Then it is a simple exercise to check that the inclusion $I_s\subset I_{s+1}$ is strict for all $s$. \label{foot:not-noetherian}} when $R$ is perfect. Hence a recurrent idea will be to pass to the \emph{perfection of $R$}. Recall that $R$ is called \emph{perfect} if the $p$-th power map on $R$ is bijective, or equivalently, if the $q$-th power map is bijective. There exists a functor $R\mapsto R^{\mathrm{perf}}$ defined as the left adjoint of the inclusion $\{\text{perfect}~\mathbb{F}\text{-algebras}\}\subset \{\mathbb{F}\text{-algebras}\}$. This means in formula that for any perfect $S$, 
\begin{equation}\label{eq:right-Hom-formula}
\Hom(R,S)=\Hom(R^{\mathrm{perf}},S).\footnote{As already mentioned in the introduction, all unlabelled $\Hom$'s are meant to be over $\FF$.}
\end{equation} By definition, $R^{\mathrm{perf}}$ is a perfect $\mathbb{F}$-algebra and there is a canonical map $R\to R^{\mathrm{perf}}$ (as coming from the identity of $R^{\mathrm{perf}}$). We begin with a well-known alternative description of $R^{\mathrm{perf}}$.
\begin{Lemma}
In the category of $\mathbb{F}$-algebras, we have a canonical isomorphism
\begin{equation}\label{eq:Rperf-as-colim}
R^{\mathrm{perf}}\cong \mathrm{colim}\left(R\xrightarrow{x\mapsto x^q} R\xrightarrow{x\mapsto x^q} R \xrightarrow{x\mapsto x^q}\cdots \right).
\end{equation}
\end{Lemma}
\begin{proof}
Let $R'$ denote the right-hand side of \eqref{eq:Rperf-as-colim}. First note that $R'$ is perfect by construction so it suffices to show that $R'$ satisfies the right Hom-formula \eqref{eq:right-Hom-formula}. We have
\[
\Hom(R',S) =\Hom(\mathrm{colim}_{x\mapsto x^q} R, S) = \mathrm{lim}_{x\mapsto x^q} \Hom(R,S)
\]
Now, if $S$ is perfect, the map $\Hom(R,S)\to \Hom(R,S)$ obtained by precomposing with the $q$th power map on $R$ is bijective and the limit above reduces to $\Hom(R,S)$.
\end{proof}

\begin{Lemma}\label{lem:reduced}
The following are equivalent:
\begin{enumerate}[label=$(\roman*)$]
\item\label{item:reduced} $R$ is reduced;
\item\label{item:injective-Frob} The $q$-th Frobenius map $\mathrm{F}:x\mapsto x^q$ on $R$ is injective;
\item\label{item:RtoRperf-injects} The map $R\to R^{\mathrm{perf}}$ is injective.
\end{enumerate}
In addition, the kernel of $R\to R^{\mathrm{perf}}$ is the nil-radical of $R$.
\end{Lemma}
\begin{proof}
That \ref{item:reduced} implies \ref{item:injective-Frob} is clear. For \ref{item:injective-Frob} implies \ref{item:reduced}, let $n$ be an arbitrary positive integer and let $h$ be such that $q^h \geq n$. Since the composition of injective maps is again injective, $x\mapsto x^{q^h}$ is injective, hence $x^n=0$ implies $x^{q^h}=0$ implies $x=0$.\\
For point \ref{item:RtoRperf-injects}, we recall that a map from $V$ to a direct colimit $U:=\mathrm{colim}_{i\geq 0} U_i$ corresponds to a map $V\to U_i$ for some $i\geq 0$. In addition, the map $V\to U$ is injective if, and only if, all compositions $V\to U_i\to U_j$ for $j\geq i$ are injective as well. In particular, the equivalence between \ref{item:injective-Frob} and \ref{item:RtoRperf-injects} is clear.\\
We prove the last statement. If $x\in R^{\mathrm{nil}}$ then $x$ becomes zero in $R^{\mathrm{perf}}$ as $x^{q^h}=0$ for some $h\geq 0$, and by definition of the colimit. Conversely, let $x\in R$ which becomes zero in $R^{\mathrm{perf}}$. In particular $x$ is zero in $S^{\mathrm{perf}}$ with $S:=R/R^{\mathrm{nil}}$. Since $S$ is reduced, $S\to S^{\mathrm{perf}}$ is injective and $x$ is zero in $S$. Hence $x\in R^{\mathrm{nil}}$.
\end{proof}

The following is proven in \cite[Lemma 3.4$(xii)$]{bhatt-scholze}.
\begin{Lemma}\label{lem:perf-respects-faithfulllyflat}
Let $R\to S$ be a faithfully flat ring homomorphism. Then $R^{\mathrm{perf}}\to S^{\mathrm{perf}}$ is also faithfully flat.
\end{Lemma}

\begin{Definition}
We let $R[\tau,\tau^{-1}]$ be the quotient of the free algebra over $R$ generated by two formal variables $\tau$, $\sigma$, subject to the relations $\tau a=a^q\tau$, $\sigma a^q=a\sigma$ and $\tau \sigma=\sigma \tau=1$, for all $a\in R$.\footnote{The experienced reader will notice that this ring is isomorphic to the left Ore localization of $R[\tau]$ at the left Ore set $S=\{1,\tau,\tau^2,\ldots \}$. Hence, our notation using $\tau^{-1}$.}
\end{Definition}

We record some immediate consequences of these relations.
\begin{Lemma}\label{lem:rewriting}
The following holds in $R[\tau,\tau^{-1}]$.
\begin{enumerate}
\item\label{item:rewriting} For all $a\in R$, $i,j\geq 0$, we have $\sigma^j a\tau^i =\sigma^{j+1}a^q \tau^{i+1}$.
\item Any element can be written in the form $\sigma^n\cdot (a_0+a_1\tau+\ldots+a_m\tau^m)$ for suitable positive integers $n,m\in \bN$ and coefficients $a_i\in R$.
\item \label{item:nilpotent-vanishes} A nilpotent element in $R$ becomes zero in $R[\tau,\tau^{-1}]$.
\end{enumerate}
\end{Lemma}

\begin{proof}
By the given relations, for all $a\in R$, $i,j\geq 0$, we have
\begin{equation}
\sigma^j a\tau^i = \sigma^j (\sigma\tau)a\tau^i = \sigma^{j+1}(\tau a)\tau^i=\sigma^{j+1}a^q \tau^{i+1}.
\end{equation}
This shows the first part.\\
Further, we observe that any monomial can be rewritten as $\sigma^j a\tau^i$ for some $i,j\in \bN$, $a\in R$. Rewriting any element $f\in R[\tau,\tau^{-1}]$ as a sum of terms $\sigma^j a\tau^i$, taking $n$ to be the maximum of all $j$ that occur in the sum, and further using point \ref{item:rewriting} to increase the $\sigma$-powers up to $n$, we obtain the desired representation.\\
Finally, using \ref{item:rewriting} again, we see that any nilpotent element $a\in R$ becomes $0$ in $R[\tau,\tau^{-1}]$.
\end{proof}

As a consequence, we obtain the following surprising proposition.
\begin{Proposition}\label{prop:inverse-tau-perfectize}
The canonical map $R[\tau,\tau^{-1}]\to R^\perf[\tau,\tau^{-1}]$ is an isomorphism.
\end{Proposition}

\begin{proof}
By Lemma \ref{lem:rewriting}\eqref{item:nilpotent-vanishes}, the nil-radical $R^{\mathrm{nil}}$ of $R$ is contained in the kernel of $R\to R[\tau,\tau^{-1}]$. In particular $R[\tau,\tau^{-1}]\cong (R/R^{\mathrm{nil}})[\tau,\tau^{-1}]$ and, since $R/R^{\mathrm{nil}} \to R^\perf$ is injective, the induced map $R[\tau,\tau^{-1}]\to R^\perf[\tau,\tau^{-1}]$ is injective as well.\\
The family of homomorphisms $g_i:R\to R[\tau,\tau^{-1}], a\mapsto \sigma^i a\tau^i$ satisfies the condition
\[ g_{i+1}(a^q)=\sigma^{i+1} a^q\tau^{i+1}\stackrel{\eqref{item:rewriting}}{=} \sigma^i a\tau^i = g_i(a) \]
for all $a\in R$, $i\geq 0$.
Hence, this family represents an element of
\[
\lim_{x\mapsto x^q} \Hom(R,R[\tau,\tau^{-1}]) \cong \Hom(R^{\perf},R[\tau,\tau^{-1}]),
\]
\emph{i.e.}~a homomorphism $g:R^\perf\to R[\tau,\tau^{-1}]$. It is easy to check that the composition $R^\perf\to R[\tau,\tau^{-1}]\to R^\perf[\tau,\tau^{-1}]$ is the canonical embedding of $R^\perf$. Therefore, the extension of $g$ to a homomorphism $R^\perf[\tau,\tau^{-1}]\to R[\tau,\tau^{-1}]$ is a left inverse to the map $R[\tau,\tau^{-1}]\to R^\perf[\tau,\tau^{-1}]$.
\end{proof}

\subsection{Formal power series in $\sigma$}
In virtue of Proposition \ref{prop:inverse-tau-perfectize}, we may and do assume that $R$ is perfect for what follows. As the $q$-power map on $R$ is bijective in that case, we use the convention $\tau^{-i}=\sigma^i$ for all $i\geq 0$, so that the relation
$\tau^i a= a^{q^i}\tau^i$ holds for all $i\in \bZ$.

We begin with the following classical lemma.
\begin{Lemma}\label{lem:coefficient}
Any element of $R[\tau,\tau^{-1}]$ can be uniquely written in the form $\sigma^s a_{-s}+\ldots+\sigma a_{-1}+a_0+\tau a_1+\ldots+\tau^r a_r$ for some $r,s\geq 0$ and coefficients $a_i\in R$.
\end{Lemma}
\begin{proof}
Consider the group $\mathcal{R}:=\bigoplus_{i\in \mathbb{Z}} \rho^i R$ where $(\rho^i)_{i\in \mathbb{Z}}$ are formal coordinates; we give a (associative, non-commutative) ring structure to $\mathcal{R}$ by setting:
\begin{equation}\label{eq:multiplication-rule}
\left(\psum_{i\in \mathbb{Z}} \rho^i a_i\right)\cdot \left(\psum_{j\in \mathbb{Z}}{\rho^j b_j}\right):=\psum_{k\in \mathbb{Z}}{\rho^k\left(\psum_{i+j=k}{a_i^{q^{-j}}b_j}\right)}
\end{equation}
where the $'$ indicates that the sum is finite.
Note that this is well-defined as $R$ is perfect. There is a map from the group freely generated by finite products of two formal variables $\tau$, $\sigma$ and elements of $R$ to $\mathcal{R}$ by mapping $\tau$ to $\rho^{1}$ and $\sigma$ to $\rho^{-1}$, elements of $R$ to $\rho^0R$ and formal products of those to the product of their image in $\mathcal{R}$. It factors through $R[\tau,\tau^{-1}]\to \mathcal{R}$ as the relations $\tau a=a^q\tau$, $\sigma a^q=a\sigma$ and $\tau \sigma=\sigma \tau=1$, for all $a\in R$, hold in $\mathcal{R}$.
To prove the lemma it suffices to show that this map is injective; we in fact show that this is an isomorphism. Indeed, there is also a map $\mathcal{R}\to R[\tau,\tau^{-1}]$ obtained by sending $\rho^i$ to $\tau^{i}$ if $i\geq 0$ and to $\sigma^{-i}$ if $i< 0$ and it is easily shown that those maps are mutual inverse to each other.
\end{proof}

\begin{Definition}\label{def:coeffn}
Let $i\in \mathbb{Z}$. We denote by $\coeff_i:R[\tau,\tau^{-1}]\to R$ the right $R$-linear map that extracts the coefficient of $\tau^i$ with respect to the presentation given in Lemma \ref{lem:coefficient} (with the convention that $\tau^{i}=\sigma^{-i}$ for $i<0$).\\
For $f\in R[\tau,\tau^{-1}]$ non zero, we denote by $\deg_\tau(f)$ the maximal integer $i$ for which $\coeff_i(f)\neq 0$. We convient that $\deg_\tau(0)=-\infty$.
\end{Definition}

\begin{Remark}\label{rem:convention-coeff}
We warn the reader that for $i\ne 0$, the map $\coeff_i$ really uses that we write the coefficients on the right, and not on the left. 
The map $\coeff_0$, however, does not depend on that convention, and it is even bi-$R$-linear.

With our convention, the maps $\coeff_i$ can also be expressed via the formula
\[
\coeff_i(f)=\coeff_0(\tau^{-i} f)
\]
for all $f\in R[\tau,\tau^{-1}]$, and any $i\in \bZ$.
\end{Remark}

The following properties of the degree in $\tau$ are easily shown.
\begin{Proposition}
Let $f,g\in R[\tau,\tau^{-1}]$.
\begin{enumerate}
\item $\deg_\tau(\tau)=1$ and $\deg_{\tau}(\sigma)=-1$;
\item We have $\deg_\tau(fg)\leq \deg_{\tau}(f)+\deg_\tau(g)$ with equality if $R$ is a domain;
\item We have $\deg_{\tau}(f+g)\leq \max\{\deg_{\tau}(f),\deg_{\tau}(g)\}$.
\end{enumerate}
\end{Proposition}

From this proposition, we see that the application $|\cdot|_{\tau}:R[\tau,\tau^{-1}]\to \mathbb{Q}_+$, $f\mapsto q^{\deg_{\tau}(f)}$ defines an ultrametric (submultiplicative) norm on $R[\tau,\tau^{-1}]$, hence a topology on $R[\tau,\tau^{-1}]$ for which the addition and multiplication are continuous operations. We also notice that the maps $\coeff_i:R[\tau,\tau^{-1}]\to R$ are continuous with respect to this topology on the source and the discrete topology on $R$.

\begin{Definition}
We define $R\ls{\sigma}$ as the noncommutative ring corresponding to the completion of $R[\tau,\tau^{-1}]$ with respect to the topology induced by $|\cdot|_{\tau}$. We still denote by $\deg_\tau$ and $\coeff_i:R\ls{\sigma}\to R$ their continuous extension.\\
We denote by $R\ps{\sigma}$ the subring of $R\ls{\sigma}$ consisting of elements of norm $\leq 1$.
\end{Definition}

\begin{Remark}\label{rem:sigma-expansion}
Following the same argument as the one given in the proof of Lemma \ref{lem:coefficient}, one shows that any element $f$ of $R\ls{\sigma}$ admit a unique expansion as
\[
f=\sum_{i\geq r} \sigma^i a_{-i}
\]
for some $r\in \bZ$ and coefficients $a_{-i}\in R$, $a_{-r}\ne 0$. Then, in fact $a_{-i}=\coeff_{-i}(f)$ for all $i$, and $-r=\deg_\tau(f)$.

Further, $f\in R\ps{\sigma}$, if and only if $r\geq 0$. 
\end{Remark}

\section{Anderson modules, motives and comotives}\label{sec:objects-from-FF-arithmetic}
This is the usual mandatory section where we recall notations and definitions of Anderson's \emph{$A$-modules}, \emph{$A$-motives} and \emph{dual $A$-motives}. We allow ourselves to generalize slightly the usual setting to permit general $A$-algebras--instead of perfect $A$-fields solely--in the definition of dual $A$-motives; concurrently we will attempt to popularize the naming \emph{$A$-comotives} instead of the usual terminology \emph{dual $A$-motives} which causes confusion with \emph{duals of $A$-motives}; the prefix ``co'' is here to reminisce about ``cocharacters''.

\subsection{Preliminaries on forms of $\mathbb{G}_a^d$}\label{subsec:forms-of-Gad}
In this minimal subsection, we recall some known facts and warnings on forms of $\mathbb{G}_a^d$. To start with, let $R$ be a ring of characteristic $p>0$ and let $G$ be a smooth affine $\mathbb{F}$-vector space scheme over $R$. We consider the following two sets of $\mathbb{F}$-vector space schemes homomorphisms over $R$:
\[
M(G):=\Hom_{\mathbb{F}/R}(G,\mathbb{G}_a), \quad \text{and}\quad N(G):=\Hom_{\mathbb{F}/R}(\mathbb{G}_a,G).
\]
The former is a left $R[\tau]$-module and the latter is a right $R[\tau]$-module through the identification of $R[\tau]$ with $\End_{\mathbb{F}}(\mathbb{G}_a)$.\\

The action of $\tau$ can be linearized in both situations. Write $\tau^* M$ for the pullback of $M=M(G)$ by the $q$-Frobenius on $R$; \emph{i.e.} $\tau^*M=R\otimes_{R,\Frob_q}M$ where the left copy of $R$ is treated as an $R$-bimodule acting by itself on the left and by the $q$th power on the right. Dually, write $\tau_* N$ for the pushforward of $N$ by the $q$-Frobenius; \emph{i.e.} $\tau_*N=N$ as groups but the $R$-module structure on the latter is ruled by $n\cdot r:=nr^q$, for $n\in N$ and $r\in R$. Then the left $R[\tau]$-module structure on $M$, respectively the right $R[\tau]$-module structure on $N$, amount\footnote{Conversely, the left action of $\tau$ is recovered by the composition $M\to \tau^*M \to M$ where the former is $m\mapsto 1\otimes m$, and the right action of $\tau$ by $N\to \tau_* N\to N$ where the former is the identity in groups.} to $R$-linear maps
\begin{equation}\label{eq:linearized-action-of-tau}
\tau^* M \longrightarrow M, \quad \text{respectively} \quad \tau_*N\longrightarrow N
\end{equation}
induced from $m\mapsto \tau m$ and $n\mapsto n\tau$. \\

Because left $R[\tau]$-modules are better behaved than right ones (\emph{e.g.}~footnote~\ref{foot:not-noetherian}), $M$ is generally a better invariant than $N$. 
\begin{Remark}\label{rem:rosenlicht-example}
To show how $N$ could be ill-behaved, we recall Rosenlicht's example. Let $R=k$ be an imperfect field, choose $a\in k\setminus k^q$, and set $k':=k(a^{1/q})$. Let $G$ be the $k$-subgroup of $\mathbb{G}_a\times \mathbb{G}_a$ consisting of couples $(x,y)$ subject to the relation $x^q=y+ay^q$. Then $G$ is not isomorphic to $\mathbb{G}_{a}$ but its base change along the inclusion $k\to k'$ is, via the map
\[
\mathbb{G}_{a,k'}\longrightarrow G_{k'}, \quad t\longmapsto (t+a^{1/q} t^q,t^{q}).
\]
As a consequence $N(G)=(0)$--as any non zero quotient of $\mathbb{G}_{a,k}$ is isomorphic to $\mathbb{G}_a$--but $N(G_{k'})\cong k'[\tau]$. In particular $N$ does not commute with flat base change in general, although $M$ does (see Remark~3.3 in \cite{hartl}).
\end{Remark}

We say that $G$ is \emph{a form of $\mathbb{G}_a^d$} if there exists a faithfully flat ring homomorphism $f:R\to S$ such that the base change $G_S$ of $G$ along $f$ is isomorphic to $\mathbb{G}_{a,S}^d$; in the latter case, we say that \emph{$G$ splits over $S$}. In the previous remark, Rosenlicht's example is a form of $\mathbb{G}_a$ which splits over $k'$. \\

From now on, we assume that $G$ is a form of $\mathbb{G}_a^d$. In this situation, $G$ can be recovered from $M$ and there is an obvious relation between $M$ and $N$.
\begin{Proposition}\label{prop:comotive-in-terms-of-motive}
Regard $R$ as a left $R[\tau]$-module where $\tau$ acts by the $q$th power. 
\begin{enumerate}
\item The map $G(R)\to \Hom_{\mathrm{left}~R[\tau]}(M(G),R)$, $g\mapsto (m\mapsto m(g))$ is an isomorphism of $\bF$-vector spaces.
\item The map $N(G)\to \Hom_{\mathrm{left}~R[\tau]}(M(G),R[\tau])$, $n\mapsto (m\mapsto m\circ n)$ is an isomorphism of right $R[\tau]$-modules.
\end{enumerate}
\end{Proposition}
\begin{proof}
The former assertion is well-known when $G=\mathbb{G}_a^d$, as then $M(G)\cong R[\tau]^d$ as left $R[\tau]$\nobreakdash-modules. In general, fix a faithfully flat ring homomorphism $R\to S$ over which $G$ splits and consider the left exact sequence of $R$-modules
\begin{equation}\label{eq:cech-nerve-flat}
0\longrightarrow  R\longrightarrow S\longrightarrow S\otimes_R S
\end{equation}
where the latter map is given by $s\mapsto s\otimes 1-1\otimes s$. Applying the (left exact) functor $\Hom_{\mathrm{left}~R[\tau]}(M(G),-)$, where \eqref{eq:cech-nerve-flat} is treated as a sequence of left $R[\tau]$-modules where $\tau$ acts via the $q$th power, we get an exact sequence
\[
0\longrightarrow \Hom_{\mathrm{left}~R[\tau]}(M(G),R)\longrightarrow \Hom_{\mathrm{left}~R[\tau]}(M(G),S)\longrightarrow  \Hom_{\mathrm{left}~R[\tau]}(M(G),S\otimes_R S)
\]
The third arrow corresponds to $G(S)\to G(S\otimes_R S)$, via the map $g\mapsto (m\mapsto m(g))$, and its kernel is $G(R)$ by the sheaf property of $G$ for the flat topology. The former assertion follows.

We turn to the second one. The assignments $S\mapsto N(G_S):=\Hom_{\mathbb{F}/S}(\mathbb{G}_{a,S},G_S)$ and $S\mapsto M(G_S)$ are well-known to be sheaves for the flat topology. We claim that
\begin{equation}\label{eq:hom-sheaf}
S\longmapsto \Hom_{\mathrm{left}~S[\tau]}(M(G_S),S[\tau])
\end{equation}
is one as well. To see this, right-tensor the sequence \eqref{eq:cech-nerve-flat} along the flat map $R\to R[\tau]$ and apply the (left exact) functor $\Hom_{\mathrm{left}-R[\tau]}(M(G),-)$ to obtain the exact sequence
\[
0\longrightarrow \Hom_{\mathrm{left}~R[\tau]}(M(G),R[\tau])\longrightarrow \Hom_{\mathrm{left}~R[\tau]}(M(G),S[\tau])\longrightarrow  \Hom_{\mathrm{left}~R[\tau]}(M(G),S\otimes_R S[\tau]).
\]
By tensor-hom adjunction and the fact that $M$ commutes with flat base change, this sequence coincides with the descent sequence for \eqref{eq:hom-sheaf} along $R\to S$. Therefore \eqref{eq:hom-sheaf} is a sheaf for the flat topology.

From this claim we may argue by descent and thus it suffices to prove the statement under the assumption that $G$ is isomorphic to $\mathbb{G}_a^d$, which is then clear.
\end{proof}

\begin{Remark}
Note however that the map $M(G)\to \Hom_{\mathrm{right}~R[\tau]}(N(G),R[\tau])$ is generally not an isomorphism, \emph{e.g.} for Rosenlicht's example. This is because the target, as opposed to \eqref{eq:hom-sheaf}, does not satisfy descent.
\end{Remark}

\begin{Proposition}\label{prop:motive-finite-presentation-flat}
$M(G)$ is of finite presentation as a left $R[\tau]$-module.
\end{Proposition}
\begin{proof}
Let $R\to S$ be a faithfully flat ring homomorphism on which $G$ splits. Since $S\otimes_R R[\tau]\cong S[\tau]$, we have as left $S[\tau]$-modules:
\[
S[\tau]\otimes_{R[\tau]} M(G) \cong S\otimes_R M(G)\cong M(G_S)\cong S[\tau]^d. 
\]
The classical descent argument carries over to the non-commutative situation (see \href{https://stacks.math.columbia.edu/tag/058Q}{[Stack Project: 058Q]}, assertion (2)). 
\end{proof}

\begin{Remark}\label{rem:motive-not-projective}
Note that $M(G)$ need not be flat over $R[\tau]$, unlike the classical descent over commutative rings would have implied (see \href{https://stacks.math.columbia.edu/tag/058Q}{[Stack Project: 058Q]}, assertion (3)). Indeed, if $M(G)$ were flat it would be projective by \cite[Theorem~4.30]{lam}, and the functor $\Hom_{\mathrm{left}~R[\tau]}(M(G),-)$ from left $R[\tau]$-modules to groups would be right exact. This functor sends the surjective morphism of left $R[\tau]$-modules $R[\tau]\to R$, $\tau\mapsto 1$, to the group homomorphism $N(G)\to G(R)$, $(n:\mathbb{G}_a\to G)\mapsto n(1_R)$ (this identification uses both points in Proposition \ref{prop:comotive-in-terms-of-motive}). In particular, the flatness of $M(G)$ would imply the surjectivity of $N(G)\to G(R)$.

We thank the referee for pointing out that this surjectivity may fail over imperfect $\mathbb{F}$-algebras $R$, as shown by Rosenlicht's example. For instance, in the notation of Remark \ref{rem:rosenlicht-example}, taking $R=\mathbb{F}(t)$ and $a=-t^{1-q}$ yields the point $(0,t)\in G(R)$, even though $N(G)$ vanishes.
\end{Remark}

\begin{Remark}\label{rem:G-Zariski-implies-M-flat}
$M(G)$ is flat over $R[\tau]$ if, however, $G$ is a Zariski form of $\mathbb{G}_a^d$. We sketch the argument here. Consider a finite Zariski covering of $\Spec R$ by opens $D(f_i)$, $i\in \{1,\ldots,n\}$, $f_i\in R$. For $I\subseteq \{1,\ldots,n\}$, write $f_I=\prod_{i\in I}{f_i}$ and $R_{f_I}$ for the localization $R[f_I^{-1}]$. Set also $M_I:=R_{f_I}\otimes_R M(G)$ so that, whenever $I$ is non empty, $M_I\cong R_{f_I}[\tau]^d$. 

We claim that $R_{f_I}[\tau]$ is a flat left $R[\tau]$-module: note that as $R\to R_{f_I}$ is relatively perfect\footnote{We call an homomorphism $R\to R'$ between characteristic $p$ rings \emph{relatively perfect} if the relative Frobenius $\Frob_*(R)\otimes_R R'\to \Frob_*(R')$ is an isomorphism. This happens for instance if $R'$ is a weakly étale $R$-algebra by \href{https://stacks.math.columbia.edu/tag/0F6W}{[Stack Project: 0F6W]}, \emph{e.g.} for finite étale maps, localizations, or maps between perfect rings, but typically not for adjoining $p$th power roots. For relatively perfect morphisms $R\to R'$, the canonical map $R[\tau]\otimes_R R'\to R'[\tau]$ is an isomorphism (note that $R'\otimes_R R[\tau]\cong R'[\tau]$ always holds). \label{ftn:relatively-perfect}}, the canonical map $R[\tau]\otimes_{R} R_{f_I}\to R_{f_I}[\tau]$ is an isomorphism. Therefore, on the category of right $R[\tau]$-modules, the functors $-\otimes_{R[\tau]} R_{f_I}[\tau]$ and $-\otimes_R R_{f_I}$ are equivalent. The latter is exact by flatness of localization, hence so is the former.

We deduce that $M_I$ itself is flat whenever $I$ is non empty. By \href{https://stacks.math.columbia.edu/tag/0G6K}{[Stack Project: 0G6K]}, we have a long exact sequence
\[
0\longrightarrow M(G)\longrightarrow \bigoplus_{i} M_{\{i\}} \longrightarrow \bigoplus_{i<j} M_{\{i,j\}} \longrightarrow \cdots \longrightarrow M_{\{1,\ldots,n\}} \longrightarrow 0.
\]
and, iterating on \href{https://stacks.math.columbia.edu/tag/00HM}{[Stack Project: 00HM]}, we deduce that $M(G)$ itself is flat. 
\end{Remark}

\begin{Corollary}\label{cor:N-commutes-with-flat-perfect-base-change}
The formation of $N$ commutes with flat base change among perfect rings. If $R$ is perfect, then $N(G)$ is finitely generated as a right $R[\tau]$-module.
\end{Corollary}
\begin{proof}
Let $R\to S$ be a flat $\mathbb{F}$-algebra homomorphism. We have
\[
N(G_S)\cong \Hom_{S[\tau]}(M(G_S),S[\tau])\cong \Hom_{R[\tau]}(M(G),S[\tau]).
\]
By \cite[X\S1 Proposition~8.(b)]{bourbaki}, the latter is isomorphic to $\Hom_{R[\tau]}(M(G),R[\tau])\otimes_{R[\tau]} S[\tau]$ whenever $S[\tau]$ is flat as a left $R[\tau]$-module. By the argument given in Remark \ref{rem:G-Zariski-implies-M-flat}, this happens if $R\to S$ is relatively perfect (see footnote \ref{ftn:relatively-perfect}); \emph{e.g.} if both $R$ and $S$ are perfect. This proves the first assertion.

For the second assertion, we may choose $S$ perfect over which $G$ splits. Since $R[\tau]\to S[\tau]$ is left-faithfully flat, we have $N(G)\otimes_{R[\tau]}S[\tau]\cong S[\tau]^d$ by the above computation. We may now descend as in proof of Proposition \ref{prop:motive-finite-presentation-flat}.
\end{proof}

\subsection{Motives and comotives}
Let $R$ be an $\mathbb{F}$-algebra (non-necessarily perfect) and let $\iota:A\to R$ be an $\mathbb{F}$-algebra homomorphism; we will voluntarily forget $\iota$ from notations. We consider the ring $A\otimes R$ where, as mentioned in the introduction, unlabelled tensor products are taken over $\mathbb{F}$. We let $\mathfrak{j}\subset A\otimes R$ be the ideal defined as the kernel of the multiplication map $A\otimes R\to R,~a\otimes r\mapsto \iota(a)r$. We let $\tau$ be the unique $A$-algebra endomorphism of $A\otimes R$ acting on $R$ as the $q$th power.
\begin{Definition}
An \emph{effective $A$-motive over $R$} is a pair $(M,\tau_M)$ where $M$ is a finite projective $A\otimes R$-module of constant rank and $\tau_M:\tau^*M\to M$ is an $A\otimes R$-linear map
whose cokernel is $\mathfrak{j}$-power torsion.

An \emph{effective $A$-comotive over $R$} is a pair $(N,\tau_N)$ where $N$ is a finite projective $A\otimes R$\nobreakdash-module of constant rank and $\tau_N$ is an $A\otimes R$-linear map $\tau_*N\to N$ whose cokernel is $\mathfrak{j}$-power torsion.

A morphism $(M_1,\tau_{M_1})\to (M_2,\tau_{M_2})$ of effective $A$-motives is an $A\otimes R$-linear map of the underlying modules $f:M_1\to M_2$ such that $\tau_{M_2}\circ \tau^*f=f\circ \tau_{M_1}$. Similarly, a morphism $(N_1,\tau_{N_1})\to (N_2,\tau_{N_2})$ of effective $A$-comotives is an $A\otimes R$-linear map of the underlying modules $g:N_1\to N_2$ such that $\tau_{N_2}\circ \tau_* g=g\circ \tau_{N_1}$.
\end{Definition}

A large supply of $A$-motives (and respectively $A$-comotives) are obtained from \emph{abelian (respectively $A$-finite) Anderson $A$-modules} whose definition was recalled as Definition \ref{def:anderson-modules} above. Let $E$ be an Anderson $A$-module over $R$. We consider the following two groups of $\mathbb{F}$-vector space schemes homomorphisms over $R$, following subsection \ref{subsec:forms-of-Gad}:
\[
M(E):=\Hom_{\mathbb{F}/R}(E,\mathbb{G}_a) \quad \text{and} \quad N(E):=\Hom_{\mathbb{F}/R}(\mathbb{G}_a,E).
\] 
Both are naturally $A\otimes R$-modules where $A$ acts on $E$ and $R$ acts on $\mathbb{G}_a$. They also admit an action of the $q$-Frobenius $\Frob_{q}$ acting on $\mathbb{G}_a$ that we denote by $\tau$. While $M(E)$ defines a left $R[\tau]$-module, $N(E)$ defines a right $R[\tau]$-module. Accordingly, as in \eqref{eq:linearized-action-of-tau}, we obtain $A\otimes R$-linear maps
\begin{equation}\label{eq:linearized-action-of-tau-Anderson}
\tau_M:\tau^*M(E)\longrightarrow M(E) \quad \text{and} \quad \tau_N:\tau_*N(E)\longrightarrow N(E).
\end{equation}
In addition, Condition \ref{item:Lie-compatible} of Definition \ref{def:anderson-modules} ensures that their cokernel are $\mathfrak{j}$-power torsion. However $(M(E),\tau_M)$, respectively $(N(E),\tau_N)$, is not yet an $A$-motive, respectively an $A$-comotive, as the finite projective condition may not be fulfilled. 
\begin{Definition}
We say that $E$ is \emph{abelian} if $M(E)$ is finite projective of constant rank over $A\otimes R$. We say that $E$ is \emph{coabelian} (or \emph{$A$-finite}) if $N(E)$ is finite projective of constant rank over $A\otimes R$.
\end{Definition}

One goal of this work is to show that these two notions are equivalent whenever $R$ is perfect.\\

We end this section with a well-known result of Anderson.
\begin{Lemma}\label{lem:anderson-exact-sequence}
Assume that $R$ is perfect and let $E$ be an Anderson $A$-motive over $R$. Then, there is an exact sequence of $A$-modules
\[
0\longrightarrow N(E)\xrightarrow{\id-\tau} N(E)\longrightarrow E(R)\longrightarrow 0.
\]
\end{Lemma}

\begin{proof}
Let $R\to S$ be a faithfully flat map of perfect $\mathbb{F}$-algebras on which $E$ splits as a form of $\mathbb{G}_a^d$; the existence of this map is ensured by Lemma \ref{lem:perf-respects-faithfulllyflat}. Corollary \ref{cor:N-commutes-with-flat-perfect-base-change} and \href{https://stacks.math.columbia.edu/tag/03OA}{[Stack Project:~03OA]} imply that 
\begin{equation}\label{eq:descent-comotive}
0\longrightarrow N(E) \longrightarrow N(E_S) \longrightarrow N(E_{S\otimes_R S})\longrightarrow N(E_{S\otimes_R S\otimes_R S})\longrightarrow \cdots
\end{equation}
is exact. $E$ splits over the perfect algebras $S^{\otimes_R n}$ for positive $n$, and therefore the sequence $0\to N(E_{S^{\otimes n}})\to N(E_{S^{\otimes n}})\to E(S^{\otimes n})\to 0$ is exact for positive $n$ by the argument given in \cite[Proposition~2.5.8]{hartl-juschka}. We deduce by a diagram chase from the exactness of \eqref{eq:descent-comotive} that the sequence is exact for $n=0$, as desired.
\end{proof}

\begin{Remark}
As the referee pointed out, Lemma \ref{lem:anderson-exact-sequence} fails to hold over imperfect $R$ (\emph{cf.} the second paragraph of Remark \ref{rem:motive-not-projective}). This is because \eqref{eq:descent-comotive} may not be exact anymore, although the subsequence $0\to N(E) \to N(E_S) \to N(E_{S\otimes_R S})$ is.
\end{Remark}

\section{Topologies}
Throughout the section, we let $R$ be a perfect $A$-algebra and we consider an abelian Anderson $A$-module $E$ over $R$. We let $M=M(E)$ be the motive of $E$ and $N=N(E)$ its comotive. 
Recall that we previously defined a norm on $R[\tau,\tau^{-1}]$ with respect to which the completion is $R\ls{\sigma}$, where $\sigma=\tau^{-1}$.

On the other hand, we let $K=\Frac(A)$, and let $\KI$ be the completion of $K$ with respect to the $\infty$-adic topology on $K$. Its valuation ring will be denoted by $\mathcal{O}_{\infty}$, and the maximal ideal in $\mathcal{O}_\infty$ by $\mathfrak{m}_\infty$.

As in \cite[\S 3.1]{gazda}, we define the topological rings
\[
\mathcal{A}_{\infty}(R):=\varprojlim_n (\mathcal{O}_{\infty}/\mathfrak{m}_{\infty}^n\otimes R), \quad \mathcal{B}_{\infty}(R):=\KI\otimes_{\mathcal{O}_{\infty}}\mathcal{A}_{\infty}(R).
\]

In this section, we investigate the (left and right) modules
\begin{align*}
M\ls{\sigma} &:=R\ls{\sigma}\otimes_{R[\tau]}M, &\quad N\ls{\sigma} &:=N\otimes_{R[\tau]} R\ls{\sigma} \\
\text{and}\quad \mathcal{B}_{\infty}(M) &:= M\otimes_{A\otimes R} \mathcal{B}_{\infty}(R), &\quad \mathcal{B}_{\infty}(N) &:= N\otimes_{A\otimes R} \mathcal{B}_{\infty}(R).
\end{align*}
We define topologies on each of them, and show that $\mathcal{B}_{\infty}(R)$ (respecively $R\ls{\sigma}$) acts continuously on $M\ls{\sigma}$ (respectively $\mathcal{B}_{\infty}(M)$) and similarly for the comotive. Finally, we deduce that we indeed have a homeomorphism connecting both which is compatible with the $R\ls{\sigma}$-action and the $\mathcal{B}_{\infty}(R)$-action. This generalizes results of Mornev on motives of Drinfeld modules \cite[\S 7]{mornev} and of the second author \cite[Proposition 7.8]{maurischat}.

\medskip

We focus on the motive here, since we assume $E$ to be abelian.
The very same statements can be given for the comotive $N(E)$ of a coabelian Anderson $A$\nobreakdash-module $E$ by switching from left-actions to right-actions of the non-commutative rings. The proofs are similar enough so that we skip those.

\subsection{The $\sigma$-adic topology}\label{sec:sigma-top}

We start with the topology on $M\ls{\sigma}$.

Note that $M\ls{\sigma}$ is finitely generated as a left module over $R\ls{\sigma}$ by Proposition \ref{prop:motive-finite-presentation-flat}. Concurrently, $N\ls{\sigma}$ is finitely generated as a right module by Corollary \ref{cor:N-commutes-with-flat-perfect-base-change}.

\begin{Definition}
A \emph{$R\ps{\sigma}$-lattice} in $M\ls{\sigma}$ is a finitely generated $R\ps{\sigma}$-submodule $\Lambda$ of $M\ls{\sigma}$ that contains a finite generating subset of  $M\ls{\sigma}$ over $R\ls{\sigma}$.
\end{Definition}

The notion of lattices allows us to define a linear topology on $M\ls{\sigma}$. Let $\Lambda$ be an $R\ps{\sigma}$-lattice in $M\ls{\sigma}$; we call \emph{the $\sigma$-topology on $M\ls{\sigma}$} the one given by the fundamental system of open submodules $(\sigma^n \Lambda)_{n\geq 0}$. This topology does not depend on the choice of the lattice $\Lambda$, as the next lemma shows.

\begin{Lemma}\label{lem:inclusion-of-lattices}
Given two $R\ps{\sigma}$-lattices $\Lambda,\Lambda'$ in $M\ls{\sigma}$, there are integers $n_1\geq n_2$ such that
\[   \sigma^{n_1}\Lambda\subseteq \Lambda' \subseteq \sigma^{n_2}\Lambda.\]
\end{Lemma}

\begin{proof}
Since $\Lambda$ contains a finite generating subset $G$ of $M\ls{\sigma}$, every element of a generating set for $\Lambda'$ is an $R\ls{\sigma}$-linear combination of elements of $G$. Choosing $n_2$ such that all the coefficients of those linear combinations are in $\sigma^{n_2}R\ps{\sigma}$ implies $\Lambda'\subseteq \sigma^{n_2}\Lambda$. For the other inclusion, switch the roles of $\Lambda$ and $\Lambda'$ in the previous argument to obtain an integer $-n_1$ with $\Lambda\subseteq \sigma^{-n_1}\Lambda'$ so that $\sigma^{n_1}\Lambda\subseteq \Lambda'$.

The inequality $n_1\geq n_2$ is a consequence of the inclusion $\sigma^{n_1}\Lambda\subseteq  \sigma^{n_2}\Lambda$.
\end{proof}

\begin{Lemma}\label{lem:abelian-implies-sigma-lattice}
Let $a\in A\setminus \bF$. Then, there exists an $R\ps{\sigma}$-lattice $\Lambda\subset M\ls{\sigma}$ and an integer $\ell>0$ such that
\[
\text{for~all~}n\geq \ell: \quad \tau \Lambda\subset a^n\Lambda
\]
\end{Lemma}
\begin{proof}
Since $a\not\in \bF$, $A$ is finite over $\bF[a]$ and so is $A\otimes R$ over $R[a]$. Hence, $M$ is not only finitely generated over $A\otimes R$, but also finitely generated over $R[a]$. In particular, there exists a finitely generated $R$-submodule $V\subset M$ which both generates $M$ as an $R[\tau]$\nobreakdash-module and as an $R[a]$-module. Consider $\Lambda_0$ the $R\ps{\sigma}$-lattice in $M\ls{\sigma}$ generated by $V$. Let $(v_i)_{i\in I}$ be a finite $R$-generating subset of $V$. For $i\in I$, we write
\begin{equation}\label{eq:tau*s-expression}
\tau v_i =\sum_{j\in I}{p_{ij}\cdot v_j} 
\end{equation}
for coefficients $p_{ij}\in R[a]$ indexed by $I^2$. Let also $\ell:=\max_{i,j}\{\deg_a(p_{ij})\}$ and set 
\[
\Lambda=\sum_{\ell>k\geq 0}{a^k\Lambda_0}.
\]
We claim that $\ell$ and $\Lambda$ are as sought, \emph{i.e.}, that $\tau\Lambda\subset a^n\Lambda$ for all $n\geq \ell$. Indeed, let $P$ be the matrix $(p_{ij})_{i,j}$ with coefficients in $R[a]$ and indexed by $I^2$ which we write as 
\[
P=P_0+P_1 a+\ldots+P_\ell a^\ell \quad \text{for}\quad P_k\in \mathcal{M}_{I}(R).
\] 
From \eqref{eq:tau*s-expression}, we obtain
\[
(1-P_0 \sigma)\tau (v_i)_{i\in I}=(\tau-P_0) (v_i)_{i\in I} = \left(\sum_{\ell\geq k>0}{P_k a^k}\right) (v_i)_{i\in I}.
\]
As $1-P_0 \sigma$ is invertible in $\mathcal{M}_I(R[\![\sigma]\!])$, we get
\[
\tau (v_i)_{i\in I}=a(1-P_0\sigma)^{-1}\left(\sum_{\ell\geq k>0}{P_k a^{k-1}}\right)(v_i)_{i\in I}.
\]
From this expression we deduce
\begin{equation}\label{eq:tauLambda-in-aLambda0}
\tau\Lambda_0\subseteq a\Lambda
\end{equation}
and then $\Lambda=\sum_{\ell>k\geq 0}{a^k \Lambda_0}=\Lambda_0+\sum_{\ell>k> 0}{a^k \Lambda_0}\subset \tau \Lambda_0+a\sum_{\ell>k\geq 0}{a^k \Lambda_0}$ which, using \eqref{eq:tauLambda-in-aLambda0}, gives:
\begin{equation}\label{eq:Lambda-in-aLambda}
\Lambda\subseteq a\Lambda.
\end{equation}
Inductively, for $k\in \{0,\ldots, \ell\}$, we have $\Lambda\subseteq a^k\Lambda\subseteq a^\ell\Lambda$ and then for $k\in \{0,\ldots,\ell-1\}$,
\[
\tau a^k\Lambda_0=a^k \tau \Lambda_0\subseteq a^{k+1}\Lambda \subseteq a^\ell \Lambda
\]
where we used \eqref{eq:tauLambda-in-aLambda0} for the first inclusion. Using \eqref{eq:Lambda-in-aLambda} inductively, gives the result as stated. 
\end{proof}

Lemma \ref{lem:abelian-implies-sigma-lattice} has the following consequence.
\begin{Corollary}\label{cor:K-automorphisms}
For any non-zero $a\in A$, the multiplication by $a$ is an isomorphism on $M\ls{\sigma}$. In particular, the $A$-action on $M\ls{\sigma}$ extends uniquely to a $K$-action on $M\ls{\sigma}$.
\end{Corollary}

\begin{proof}
The statement is clear for $a\in \FF\setminus \{0\}$. So we assume that $a\in A\setminus \FF$. 
By assumption $M$ is a finite projective $A\otimes R$-module. In particular, the multiplication by a non zero element $a\in A$ on $M$ is an injective $R[\tau]$-linear operation. Since $R[\tau]\to R\ls{\sigma}$ is flat, the multiplication by $a$ stays injective on $M\ls{\sigma}$.\\
By Lemma \ref{lem:abelian-implies-sigma-lattice}, there exists an $R\ps{\sigma}$-lattice in $M\ls{\sigma}$ and an integer $n\geq 0$ such that $\tau\Lambda\subseteq a^n\Lambda$.
In particular, $a^n\Lambda$ contains an $R\ls{\sigma}$-generating set of $M\ls{\sigma}$. This implies that multiplication with $a^n$ is surjective, and hence also multiplication with $a$ is surjective.
\end{proof}

Next, we upgrade Lemma \ref{lem:abelian-implies-sigma-lattice}. 
\begin{Proposition}\label{prop:st-lattice}
Let $\Lambda$ be an $R\ps{\sigma}$-lattice in $M\ls{\sigma}$, and $a\in A\setminus \bF$. 
Then there is an integer $N>0$ such that
\[ \Lambda\subseteq \tau \Lambda \subseteq a^{N}\Lambda. \]
\end{Proposition}

\begin{proof}
The former inclusion is tautological as $\Lambda$ is stable by $\sigma=\tau^{-1}$. Let $\ell>0$ and $\Lambda_a$ be as in Lemma \ref{lem:abelian-implies-sigma-lattice}, namely
\[  \tau\Lambda_a \subseteq a^\ell\Lambda_a.\]
Up to replacing $\Lambda_a$ by $\tau^k\Lambda_a$ for some $k$ large enough, we may assume that $\Lambda \subseteq \Lambda_a$. Further, by Lemma \ref{lem:inclusion-of-lattices}, there is an integer $n\geq 0$ for which $\Lambda_a\subseteq \tau^n \Lambda$.
Hence,
\[  \tau^{n+1}\Lambda \subseteq \tau^{n+1}\Lambda_a \subseteq a^{(n+1)\ell}\Lambda_a\subseteq a^{(n+1)\ell}\tau^n\Lambda. \]
Multiplying the whole chain of inclusions by $\sigma^n$, gives the desired inclusion with $N=(n+1)\ell$. 
\end{proof}

\begin{Proposition}\label{prop:continuous-Binfty-action}
The $K$-action on $M\ls{\sigma}$ extends uniquely to a continuous action of $\KI$ and even a continuous action of $\mathcal{B}_{\infty}(R)$. 
\end{Proposition}

\begin{proof}
We settle some notations first. Let $\Lambda$ be an $R\ps{\sigma}$-lattice in $M\ls{\sigma}$. By Proposition \ref{prop:st-lattice}, there exists $b\in A\setminus \bF$ such that $b^{-1}\Lambda \subseteq \sigma\Lambda$.
Let $z:=b^{-1}\in \cO_\infty$ and let $g_1,\ldots, g_\ell\in \cO_\infty\cap K$ be representatives in $K$ of an $\bF$-basis of $\cO_\infty/z\cO_\infty$.
Then any $g\in \mathcal{B}_{\infty}(R)$ can uniquely be written as a convergent series
\[ g=\sum_{j=-j_0}^\infty{ z^{j}\left(\sum_{i=1}^\ell {g_i\otimes c_{ij}}\right)}  \]
for appropriate $j_0\in\ZZ$ and $c_{ij}\in R$.\\
By Corollary \ref{cor:K-automorphisms}, we know that each $g_i\in K$ acts as an automorphism on $M\ls{\sigma}$; in particular, $g_i\Lambda$ is also a $R\ps{\sigma}$-lattice in $M\ls{\sigma}$. Let $\nu$ be a large enough integer for which $g_i\Lambda \subseteq \sigma^{-\nu}\Lambda$ for all $i\in \{1,\ldots, \ell\}$. 

Back to the statement, we have to show that there exists a unique continuous dashed arrow making the following diagram commute
\[
\begin{tikzcd}
K\times M\ls{\sigma} \arrow[r,"\text{Cor.}~\ref{cor:K-automorphisms}"]\arrow[d] & M\ls{\sigma} \\
\mathcal{B}_{\infty}(R)\times M\ls{\sigma} \arrow[ur,dashed] & 
\end{tikzcd}
\]
Given such an action denoted by a dot, $m\in M\ls{\sigma}$ and $j\in \mathbb{Z}$, the expression 
\[
\left(z^{j}\sum_{i=1}^\ell {g_i\otimes c_{ij}}\right)\cdot m
\]
is uniquely determined because the element inside the parenthesis is in $K\otimes R$. For some integer $k_0$, we have $m\in \sigma^{-k_0} \Lambda$ and since
\begin{equation}\label{eq:inclusion-for-continuity}
\left(z^{j}\sum_{i=1}^\ell {g_i\otimes c_{ij}}\right)\cdot \sigma^{-k_0}\Lambda = \sigma^{-k_0}z^{j}\left(\sum_{i=1}^\ell {g_i\otimes c_{ij}}\right)\cdot \Lambda \subseteq \sigma^{-k_0}z^{j}\sigma^{-\nu}\Lambda \subseteq   \sigma^{-(\nu+k_0)+j}\Lambda,
\end{equation}
the following sum converges in $M\ls{\sigma}$:
\begin{equation}\label{eq:definition-of-dot}
\sum_{j=-j_0}^\infty{\left( z^{j}\sum_{i=1}^\ell {g_i\otimes c_{ij}}\right) \cdot m }
\end{equation}
and belongs to $\sigma^{-(\nu+k_0+j_0)}\Lambda$. Continuity of $\cdot$ enforces $g\cdot m$ to coincide with the above expression. This shows uniqueness. To prove existence, it remains to show that the assignment $(g,m)\mapsto \eqref{eq:definition-of-dot}$ is continuous; but this follows from \eqref{eq:inclusion-for-continuity} and \eqref{eq:definition-of-dot} as well.
\end{proof}

\subsection{The $\infty$-adic topology}\label{sec:infty-top}

We now change the roles of $R\ls{\sigma}$ and $\mathcal{B}_{\infty}(R)$, and
consider the scalar extension
\[ \mathcal{B}_{\infty}(M):= M\otimes_{A\otimes R}\mathcal{B}_{\infty}(R), \]
define a topology on it, and show that it admits a continuous $R\ls{\sigma}$-action extending the $R[\tau]$-action.

As before, we assume that $E$ is abelian, i.e.~that $M$ is finite projective over $A\otimes R$.
If we assumed that $E$ was coabelian, we would obtain the analogous statements for the comotive $N$ using right modules instead of left modules.

Note the canonical isomorphism $\mathcal{B}_{\infty}(\tau^*M)\cong \tau^* \mathcal{B}_{\infty}(M)$.

\begin{Definition}
An \emph{$\mathcal{A}_{\infty}(R)$-lattice} in $\mathcal{B}_{\infty}(M)$ is a finitely generated $\mathcal{A}_{\infty}(R)$-submodule which generates $\mathcal{B}_{\infty}(M)$ over $\mathcal{B}_{\infty}(R)$. 
\end{Definition}
\begin{Remark}
We do not require lattices to be locally-free as opposed to \cite[Definition 3.4.4]{mornev}.
\end{Remark}

Similarly, $\mathcal{A}_{\infty}(R)$-lattices are useful to define a topology. Let $\Lambda$ be an $\mathcal{A}_{\infty}(R)$-lattice in $\mathcal{B}_{\infty}(M)$; we call \emph{the $\infty$-topology on $\mathcal{B}_{\infty}(M)$} the one given by the fundamental system of open submodules $(\mathfrak{m}_{\infty}^n \Lambda)_{n\geq 0}$. This topology does not depend on the choice of the lattice $\Lambda$ as one shows in the same way as Lemma \ref{lem:inclusion-of-lattices}.\\

The connection between the $\sigma$-topology and the $\infty$-topology is made through the following proposition. 
\begin{Proposition}\label{prop:existence-of-sigmastablelattice}
Let $\Lambda$ be an $\mathcal{A}_{\infty}(R)$-lattice of $\mathcal{B}_{\infty}(M)$. Then, there exists an integer $s_0>0$ such that $\sigma^s\Lambda\subset \mathfrak{m}_{\infty}\Lambda$ for all $s\geq s_0$.
\end{Proposition}
\begin{proof}
Because $M$ comes from an Anderson $A$-module over $R$, it is finitely generated as an $R[\tau]$-module by Proposition \ref{prop:motive-finite-presentation-flat}. In particular there exists a finite $R$-submodule $V\subset M$ for which
\begin{equation}\label{eq:Mfinite}
M=V+\tau V+\tau^2V+\ldots
\end{equation}
If $\Lambda_0\subset \mathcal{B}_{\infty}(M)$ denotes a finitely generated $\mathcal{A}_{\infty}(R)$-submodule of $\mathcal{B}_{\infty}(M)$ which contains $V$, we claim that \eqref{eq:Mfinite} implies:
\begin{equation}\label{eq:full-sum-in-B(M)}
\mathcal{B}_{\infty}(M)=\Lambda_0+\tau\Lambda_0+\tau^2\Lambda_0+\ldots
\end{equation}
Indeed, as $R$ is perfect the right-hand side of \eqref{eq:full-sum-in-B(M)} is an $\mathcal{A}_{\infty}(R)$-submodule of $\mathcal{B}_{\infty}(M)$ that contains $M$. Since the subring of $\mathcal{B}_{\infty}(R)$ generated by $\mathcal{A}_{\infty}(R)$ and $A$ is $\mathcal{B}_{\infty}(R)$ itself, the right-hand side of \eqref{eq:full-sum-in-B(M)} contains the $\mathcal{B}_{\infty}(R)$-submodule generated by $M$, hence is $\mathcal{B}_{\infty}(M)$. 

In particular, there exists some integer $N_0> 0$ for which we have both
\begin{equation}\label{eq:fundamental-inclusions}
\sigma\Lambda_0\subset \Lambda_0+\tau\Lambda_0+\ldots+\tau^{N_0}\Lambda_0, \quad \text{and}\quad \mathfrak{m}_{\infty}^{-1}\Lambda_0\subset \Lambda_0+\tau\Lambda_0+\ldots+\tau^{N_0}\Lambda_0.
\end{equation}
By induction, the first inclusion implies  $\sigma^n\Lambda_0\subset \Lambda_0+\tau\Lambda_0+\ldots+\tau^{N_0}\Lambda_0$ for all $n\geq 0$. Setting
\[
\Lambda^{\mathrm{st}}_0:=\sum_{n\geq 0}{\sigma^n \Lambda_0}
\]
gives a module stable by $\sigma$ such that $\Lambda^{\mathrm{st}}_0\subset \Lambda_0+\tau\Lambda_0+\ldots+\tau^{N_0}\Lambda_0$, and which contains $\Lambda_0$. Hence it is generating, and is contained in a finitely generated $\mathcal{A}_{\infty}(R)$-submodule; note that it is not necessarily an $\mathcal{A}_{\infty}(R)$-lattice as it may not be of finite type (typically if $R$ is not noetherian). Applying $\sigma^{N_0}$ to the second inclusion of \eqref{eq:fundamental-inclusions} yields $\mathfrak{m}_{\infty}^{-1}\sigma^{N_0}\Lambda_0 \subset \Lambda_0+\sigma \Lambda_0+\ldots+\sigma^{N_0} \Lambda_0$. Applying then $\sigma^n$ and summing over all $n\geq 0$ gives $\sigma^{N_0}\Lambda^{\mathrm{st}}_0\subset \mathfrak{m}_{\infty}\Lambda^{\mathrm{st}}_0$.

By construction of $\Lambda^{\mathrm{st}}_0$, there exists two integers $k_1$ and $k_2$ for which $\mathfrak{m}_{\infty}^{k_1}\Lambda^{\mathrm{st}}_0\subset \Lambda \subset \mathfrak{m}_{\infty}^{-k_2}\Lambda^{\mathrm{st}}_0$. Set $s_0:=N_0(k_1+k_2+1)$. Then for all $s\geq s_0$,
\[
\sigma^{s}\Lambda = \sigma^{N_0(k_1+k_2+1)+(s-s_0)}\Lambda\subset \mathfrak{m}_{\infty}^{-k_2}\sigma^{N_0(k_1+k_2+1)}\sigma^{(s-s_0)}\Lambda^{\mathrm{st}}_0\subset \mathfrak{m}_{\infty}^{-k_2}\mathfrak{m}_{\infty}^{k_1+k_2+1}\Lambda^{\mathrm{st}}_0\subset \mathfrak{m}_{\infty}\Lambda.
\]
In particular, the integer $s_0$ works as desired.
\end{proof}

Note that $\tau$ acts bijectively on $\mathcal{B}_{\infty}(M)$ as is shown, \emph{e.g.}, in \cite[Proposition 3.17]{gazda}. In particular,  $\mathcal{B}_{\infty}(M)$ is canonically an $R[\tau,\tau^{-1}]$-module. With the help of Proposition \ref{prop:existence-of-sigmastablelattice} and reasoning as in the proof of Proposition \ref{prop:continuous-Binfty-action}, we obtain:
\begin{Proposition}\label{prop:continuous-R((sigma))-action}
The left-action of $R[\tau,\tau^{-1}]$ on $\mathcal{B}_{\infty}(M)$ extends uniquely to a continuous left-action of $R(\!(\sigma)\!)$.
\end{Proposition}

\subsection{Isomorphism of topological spaces}\label{sec:homeomorphism}

The next theorem is a surprising byproduct of most results of Sections \ref{sec:sigma-top} \& \ref{sec:infty-top}. It generalizes \cite[Proposition 7.8 (a)]{maurischat}.
\begin{Theorem}
The topological $R$-modules $\mathcal{B}_{\infty}(M)$ and $M\ls{\sigma}$ are canonically homeomorphic. 
\end{Theorem}
\begin{proof}

Proposition \ref{prop:continuous-Binfty-action} gives a unique continuous action of $\mathcal{B}_{\infty}(R)$ on $M\ls{\sigma}$. Hence, the universal property of the tensor product then gives a unique $\mathcal{B}_{\infty}(R)$-linear map $\iota_1$ making the following diagram commute:
\[
\begin{tikzcd}
\mathcal{B}_{\infty}(R)\times M \arrow[r]\arrow[d] & M\ls{\sigma} \\
\mathcal{B}_{\infty}(M)\arrow[ur,dashed,"\iota_1"']
\end{tikzcd}
\]
Similarly, Proposition \ref{prop:continuous-R((sigma))-action} gives a unique left $R\ls{\sigma}$-linear map $\iota_2$ inserting in a commutative diagram
\[
\begin{tikzcd}
 & M\ls{\sigma}\arrow[dl,dashed,"\iota_2"'] \\
\mathcal{B}_{\infty}(M) & R\ls{\sigma}\times M \arrow[u]\arrow[l]
\end{tikzcd}
\]
We claim that $\iota_1$ and $\iota_2$ are continuous. Indeed, let $V\subset M$ be a finite $R$-submodule which is both generating for the $R[\tau]$ and the $A\otimes R$-module actions. Let $\Lambda_{\infty}\subset \mathcal{B}_{\infty}(M)$ be the $\mathcal{A}_{\infty}(R)$-lattice it generates, and $\Lambda_{\sigma}\subset M\ls{\sigma}$ the $R\ps{\sigma}$-lattice it generates. By Proposition \ref{prop:existence-of-sigmastablelattice} there exists $N>0$ large enough for which $\sigma^N\Lambda_{\infty}\subset \mathfrak{m}_{\infty}\Lambda_{\infty}$. We then set
\[
\Lambda_{\infty}^{\mathrm{st}}:=\sum_{n=0}^{N-1}{\sigma^n \Lambda_{\infty}}
\] 
which defines another $\mathcal{A}_{\infty}(R)$-lattice in $\mathcal{B}_{\infty}(M)$ which is stable under the action of $R[\![\sigma]\!]$. In particular $\iota_2(\Lambda_{\sigma})\subset \Lambda_{\infty}^{\mathrm{st}}$ and for any positive integer $k$
\[
\iota_2(\sigma^{Nk}\Lambda_{\sigma})=\sigma^{Nk} \iota_2(\Lambda_{\sigma})\subset \sigma^{Nk}\Lambda_{\infty}^{\mathrm{st}} \subset \mathfrak{m}_{\infty}^k\Lambda_{\infty}^{\mathrm{st}}
\]
proving the continuity of $\iota_2$. 

For $\iota_1$, the argument is similar: Proposition \ref{prop:st-lattice} implies the existence of $z\in \mathcal{O}_{\infty}\cap K$ such that $z\Lambda_{\sigma}\subset \sigma\Lambda_{\sigma}$. Setting
\[
\Lambda_{\sigma}^{\mathrm{st}}:=\sum_{g\in S}{g\Lambda_{\sigma}},
\]
where $S\subset K$ is a set of representative in $K$ of the quotient $\mathcal{O}_{\infty}/z\mathcal{O}_{\infty}$, we find $\mathfrak{m}_{\infty}^N\Lambda_{\sigma}^{\mathrm{st}}\subset \sigma \Lambda_{\sigma}^{\mathrm{st}}$ where $N$ is the positive integer for which $(z)=\mathfrak{m}_{\infty}^N$ as ideals of $\mathcal{O}_{\infty}$. This implies $\iota_1(\Lambda_{\infty})\subset \Lambda_{\sigma}^{\mathrm{st}}$ and 
\[
\iota_1(\mathfrak{m}_{\infty}^{Nk}\Lambda_{\infty})=\mathfrak{m}_{\infty}^{Nk} \iota_1(\Lambda_{\infty})\subset \mathfrak{m}_{\infty}^{Nk}\Lambda_{\sigma}^{\mathrm{st}} \subset \sigma^k\Lambda_{\sigma}^{\mathrm{st}}
\]
proving the continuity of $\iota_1$. In fact, this further shows that $\iota_1$ is also $R(\!(\sigma)\!)$-linear and $\iota_2$ is $\mathcal{B}_{\infty}(R)$-linear. 

That $\iota_1$ and $\iota_2$ are mutual inverse to each other then follows immediately from the fact that $\iota_1\circ \iota_2|_{M}=\iota_2\circ \iota_1|_{M}=\id_M$.
\end{proof}

\section{The residue-in-$\tau$ pairing}\label{sec:residue-in-tau}

We again assume that the $A$-algebra $R$ is perfect. 
We let $E$ be an Anderson $A$-module over $R$, $M(E)=\Hom(E,\mathbb{G}_a)$ be its motive and $N(E)=\Hom(\mathbb{G}_a,E)$ its comotive. For better readability, we denote by $\Phi:A\to \End_{\mathbb{F}}(E)$ the $A$-module scheme structure of $E$. 

Assuming that $E$ is abelian or coabelian, we reinterpret the pairing of Hartl and Juschka \cite[Theorem 2.5.13]{hartl-juschka} and construct an $A\otimes R$-linear map
\begin{equation}\label{eq:Res-tau}
\mathrm{Res}_{\tau}:\tau^{*}M(E)\otimes_{A\otimes R} N(E)\longrightarrow \ka\otimes R. 
\end{equation}

\subsection{Construction of $\mathrm{Res}_{\tau}$}

The construction of the map $\mathrm{Res}_{\tau}$ is done in several steps.
\paragraph{Step 1.} Given $m:E\to \mathbb{G}_a$ in $M(E)$ and $n:\mathbb{G}_a\to E$ in $N(E)$, we may compose them to obtain $m\circ n
\in \End(\Ga)=R[\tau]$. Better, we obtain a map 
\begin{equation}\label{eq:composition}
M(E)\times N(E)\longrightarrow \Hom_{\bF}(A,R[\tau]), \quad (m,n)\longmapsto (a\mapsto m\circ \Phi(a) \circ n) 
\end{equation}
into the $\bF$-linear homomorphism from $A$ to $R[\tau]$. It is, by design, $A$-bilinear and hence factors uniquely through $M(E)\otimes_A N(E)$.

\paragraph{Step 2.} If $E$ is abelian we may upgrade the previous construction thanks to Corollary \ref{cor:K-automorphisms} and Proposition \ref{prop:continuous-Binfty-action}. Indeed, the composition map $M(E)\otimes_{\mathbb{F}} N(E)\to \End(\mathbb{G}_a)=R[\tau]$ promotes, extending scalars on the left along $R[\tau]\to R\ls{\sigma}$, to
\[
M(E)\ls{\sigma}\otimes_{\mathbb{F}}N(E)\longrightarrow R\ls{\sigma}.
\]
The same formula \eqref{eq:composition} allows to define a map 
\begin{equation}\label{eq:construction2}
M(E)\otimes_A N(E)\longrightarrow \Hom_{\bF}(\KI,R\ls{\sigma})
\end{equation}
assigning, to a couple $(m,n)$, the map $\KI \ni f \mapsto (m\circ \Phi(f))\circ n\in R\ls{\sigma}$ where $\Phi$ now denotes the extended action of Proposition \ref{prop:continuous-Binfty-action}. 

Respectively, if $E$ is coabelian, scalars are extended on the right and we rather consider the composition $m\circ (\Phi(f)\circ n)$.

If $E$ is abelian and coabelian, the two versions give the same map, as both are obtained by unique continuous extensions of the map \eqref{eq:composition}.

\begin{Proposition}\label{prop:continuous-pairing}
The map \eqref{eq:construction2} lands in $\Hom_{\bF}^{\mathrm{cont}}(\KI,R\ls{\sigma})$; \emph{i.e.} the submodule of continuous homomorphisms with respect to the $\infty$-adic topology on $K_{\infty}$ and the $\sigma$-adic topology on $R\ls{\sigma}$. 
\end{Proposition}
\begin{proof}
This is a consequence of the continuous action of $K_{\infty}$ on $M$. In details, let $m\in M$, $n\in N$ and let $h:K_{\infty}\to R\ls{\sigma}$ be the $\mathbb{F}$-linear homomorphism associated to $m\otimes n$ via \eqref{eq:construction2}.  We shall show that, for all $D\geq 0$ there exists $\delta_D$ such that if $g\in K_{\infty}$ with $v_{\infty}(g)\geq \delta_D$ then $h(g)\in \sigma^{D}R[\![\sigma]\!]$. 

Let $\kappa=(\kappa_1,\ldots,\kappa_d)$ be generators of $M$ over $R[\tau]$ and let $u$ be the maximal degree of the polynomials in $\tau$ appearing as coefficients in an expression of $m$ written in $\kappa$. Let $\Lambda_{\kappa}$ be the $R[\![\sigma]\!]$-lattice of $M(\!(\sigma)\!)$ generated by $\kappa$ (so that $m\in \sigma^{-u}\Lambda_\kappa$). Let also $v:=\max_i\{\deg_{\tau}(\kappa_i\circ n)\}$ so that any element of $\Lambda_\kappa$ composed with $n$ belongs to $\sigma^{-v}R[\![\sigma]\!]$.

By Proposition \ref{prop:continuous-Binfty-action}, let $\delta_D$ be such that $g\Lambda_\kappa\subset \sigma^C \Lambda_\kappa$ for all $g\in K_{\infty}$ with $v_{\infty}(g)\geq \delta_D$, where $C=u+v+D$. Then,
\[
m\circ g \circ n \in \sigma^{-u}\Lambda_\kappa\circ g\circ n \subseteq \sigma^{C-u}\Lambda_\kappa\circ n \subseteq \sigma^{C-(u+v)}R\sps{\sigma}=\sigma^{D}R\sps{\sigma}.\qedhere
\]
\end{proof}

\paragraph{Step 3.} Still, the morphism \eqref{eq:construction2} is not $R$-linear. To address this, for some integer $\ell$, we compose with the continuous map $\mathrm{coeff}_{-\ell}:R\ls{\sigma}\to R$ of Definition \ref{def:coeffn} which by Remark \ref{rem:convention-coeff} amounts to
\begin{equation}\label{eq:construction3}
M(E)\otimes_A N(E)\longrightarrow \Hom_{\bF}^{\mathrm{cont}}(\KI,R), \quad m\otimes n\longmapsto (g\mapsto \mathrm{coeff}_{0}(\tau^{\ell}\circ m\circ \Phi(g)\circ n)).
\end{equation}
The resulting map \eqref{eq:construction3} is now $R$-sesquilinear with respect to the $q^\ell$-power map on $R$, hence factors through $\tau^{\ell*}M(E)\otimes_{A\otimes R} N(E)$. While it is tempting to take $\ell=0$ to have $R$-linearity, the choice $\ell=1$ has the quite pleasant feature that any element in the image of \eqref{eq:construction3} vanishes on $A\subset \KI$. That is, for $\ell=1$, \eqref{eq:construction3} becomes:
\begin{equation}\label{eq:construction4}
\tRestau: \tau^*M(E)\otimes_{A\otimes R} N(E)\longrightarrow \Hom_{\bF}^{\mathrm{cont}}(\KI/A,R)
\end{equation}
(the quotient $\KI/A$ is endowed with the quotient topology).

\paragraph{Step 4.} The final step relies on residue-duality (\emph{e.g.} \cite[Theorem 8]{poonen}), which asserts that the pairing
\[
\ka\times (\KI/A)\longrightarrow \mathbb{F}, \quad \omega\otimes (g+A)\longmapsto \mathrm{Tr}_{\mathbb{F}_{\infty}|\mathbb{F}}(\mathrm{Res}_{\infty}(g\omega)).
\]
identifies $\ka$ as the continuous dual of $\KI/A$. Extending scalars along $\mathbb{F}\to R$, we get an $A\otimes R$-linear isomorphism
\begin{equation}\label{eq:residue-duality}
\ka\otimes_{\mathbb{F}} R \cong \Hom_{\mathbb{F}}^{\mathrm{cont}}(\KI/A,R). 
\end{equation}
The combination of \eqref{eq:construction4} and \eqref{eq:residue-duality} yields the construction of \eqref{eq:Res-tau}.
\begin{Definition}
If $E$ is abelian, we define \emph{the residue-in-$\tau$ pairing of $E$} to be the unique $A\otimes R$-linear map 
\[
\mathrm{Res}_{\tau}:\tau^*M(E)\otimes_{A\otimes R}N(E)\longrightarrow \ka\otimes_{\mathbb{F}} R
\]
assigning to an elementary tensor $m\otimes n$ the unique differential form $\omega$ which satisfies, for all $g\in K_{\infty}$, 
\[\mathrm{Tr}_{\mathbb{F}_{\infty}|\mathbb{F}}(\mathrm{Res}_{\infty}(g\omega))=\mathrm{coeff}_{0}(\tau\circ (m\circ \Phi(g))\circ n).\]  

If $E$ is coabelian, we define this map via the same formula, but with a change of parenthesis:
\[ \mathrm{Tr}_{\mathbb{F}_{\infty}|\mathbb{F}}(\mathrm{Res}_{\infty}(g\omega))=\mathrm{coeff}_{0}(\tau\circ m\circ (\Phi(g)\circ n)).\]  
\end{Definition}

In the following, we will omit the inner most parenthesis to unify presentation for both case of $E$ being abelian or $E$ being coabelian.
Indeed, we will see in Corollary \ref{cor:abelian=coabelian} that the pairing enables us to show that the properties of being abelian and being coabelian are equivalent.

\begin{Remark}\label{rem:pushforward-pairing}
Since $R$ is perfect, we have $\tau_*\tau^*M(E)\cong M(E)$. Hence, when we apply the functor $\tau_*$ to the pairing, we obtain another $A\otimes R$-linear pairing
\begin{equation}\label{eq:relation-among-two-res}
\mathrm{Res}^\tau:= \tau_* \mathrm{Res}_{\tau} : M(E)\otimes_{A\otimes R} \tau_* N(E) \longrightarrow \tau_* (\ka\otimes R)
\end{equation}
which is given by the same formula (just the $R$-actions are different on the objects). This will be used in Theorem \ref{thm:res-perfect}.
\end{Remark}

\begin{Remark}\label{rem:explicit-residue-duality}
If $C$ denotes the curve $\mathbb{P}^1$ over $\mathbb{F}$ and $\infty$ the point $[0:1]$, then $A$ identifies with $\bF[t]$ for some function $t$ on $\mathbb{P}^1$ having a simple pole at $\infty$. The field $K_{\infty}$ identifies with $\mathbb{F}\ls{\varpi}$ where we set $\varpi:=1/t$. \\
In this context, the isomorphism $\Hom_{\bF}^{\mathrm{cont}}(\KI/A,R)\to \ka\otimes R$ is explicitly given by sending a continuous $\mathbb{F}$-linear map $f:K_{\infty}/A\to R$ to the differential form
\[
-\sum_{k=0}^\infty{f\left(\varpi^{k+1}\right)t^k dt}.
\]
The continuity of $f$ ensures that this sum is finite.

The residue pairing is therefore
\[ \Restau(m\otimes n) = -\sum_{k=0}^\infty \mathrm{coeff}_{0}\left( \tau\circ m\circ \Phi(\varpi^{k+1})\circ n\right) t^k dt.\]
\end{Remark}

An immediate corollary of the construction is the following property (which is less obvious to see in \cite{hartl-juschka}).
\begin{Proposition}[Commutation with $\tau$]\label{prop:commutation-with-tau}
For all couples $(m,n)$ in $M(E)\times N(E)$, we have \[\Restau((\tau\circ m)\otimes n)=(\id\otimes \Frob_q)(\Restau(m\otimes (n\circ \tau))) \in \ka\otimes R.\]
\end{Proposition}
\begin{proof}
As the residue duality is compatible with the $\tau$-action, we can equivalently show the formula for $\tRestau$.
 
For $f\in K_{\infty}$, we write $\tau\circ m\circ \Phi(f)\circ n\in R\ls{\sigma}$ as $\sum_{i}{\sigma^ic_{-i}}$. Note that $c_0=\tRestau(m\otimes n)(f)$. Then $\tau\circ (\tau\circ m)\circ \Phi(f)\circ n$ corresponds to $\sum_{i}{\tau\sigma^ic_{-i}}$ whose zeroth coefficient is $c_{-1}$; on the other hand $\tau\circ m\circ \Phi(f)\circ n\circ \tau$ rather corresponds to $\sum_{i}{\sigma^ic_{-i}\tau}$ whose zeroth coefficient is $(c_{-1})^{\!1/q}$.
\end{proof}

\subsection{The pairing $\mathrm{Res}_{\tau}$ is perfect}
Our main result of this subsection is the following.

\newcommand{\tauF}{\tau_F}  

\begin{Theorem}\label{thm:res-perfect}
Let $R$ be a perfect $A$-algebra. Assume that $E$ is abelian (respectively coabelian). Then, the pairing $\mathrm{Res}_{\tau}$ is perfect; \emph{i.e.}~both induced maps:
\begin{align*}
\Restau(-\otimes \star) &: N\longrightarrow \Hom_{A\otimes R}\left( \tau^*M,\ka\otimes R\right), \quad n\longmapsto \Restau(-\otimes n),\quad \text{and} 
\\
\Restau(\star \otimes -) &: \tau^*M\longrightarrow \Hom_{A\otimes R}\left( N,\ka\otimes R\right), \quad m\longmapsto \Restau(m\otimes -),
\end{align*}
are isomorphisms of $A\otimes R$-modules. 
In addition, these isomorphisms are compatible with $\tau$ in the sense that the following two diagrams of $A\otimes R$-modules commute:
\[\begin{tikzcd}[column sep=3.5em]
	\tau_* N & {\Hom_{A\otimes R}(M,\tau_*(\ka\otimes R))} & {\tau^*M} & {\Hom_{A\otimes R}(N,\ka\otimes R)} \\
	N & {\Hom_{A\otimes R}(\tau^{*}M,\ka\otimes R)} & M & {\Hom_{A\otimes R}(\tau_*N,\tau_*(\ka\otimes R))}
	\arrow["{\mathrm{Res}^{\tau}(-\otimes\star)}", from=1-1, to=1-2]
	\arrow["\tau_N"', from=1-1, to=2-1]
	\arrow["{\eta\mapsto \tauF^{-1}\circ \eta\circ \tau_M}", from=1-2, to=2-2]
	\arrow["{\mathrm{Res}_{\tau}(\star\otimes -)}", from=1-3, to=1-4]
	\arrow["\tau_M"', from=1-3, to=2-3]
	\arrow["{\check{\eta}\mapsto \tauF \circ \check{\eta}\circ \tau_N}", from=1-4, to=2-4]
	\arrow["{\Restau(-\otimes\star)}", from=2-1, to=2-2]
	\arrow["{\mathrm{Res}^{\tau}(\star\otimes -)}", from=2-3, to=2-4]
\end{tikzcd}\]
Above, $\tau_M$ and $\tau_N$ are the linearized $\tau$-actions \eqref{eq:linearized-action-of-tau-Anderson}, and $\tauF$ the linear isomorphism $\ka\otimes R \xrightarrow{\sim} \tau_*(\ka\otimes R)$, induced from the $q$-Frobenius on $R$, i.e.~from $\omega\otimes r\mapsto \omega\otimes r^q$, as well as $\mathrm{Res}^{\tau}=\tau_*\Restau$ the pairing from Remark \ref{rem:pushforward-pairing}.
\end{Theorem}

\begin{Remark}
We remark, that this definition of the $\tau$-action on $\Hom_{A\otimes R}(\tau^*M,\ka\otimes R)$ is the natural way of defining a right $\tau$-action on the homomorphisms for a left $\tau$-module $\tau^*M$.
Also, the definition of the $\tau$-action on $\Hom_{A\otimes R}(\tau_*N,\tau_*(\ka\otimes R))$ is the natural way of defining a left $\tau$-action on the homomorphisms for a right $\tau$-module $\tau_*N$. This switch between left and right is the reason, why we need $\tauF^{-1}$ in one definition but $\tauF$ in the other.

That both diagrams commute is a simple reformulation of Proposition \ref{prop:commutation-with-tau}.
\end{Remark}

We will actually prove the theorem for $\tRestau$ from which the one for $\Restau$ directly follows using the isomorphism from residue duality \eqref{eq:residue-duality}.

Assuming $E$ to be abelian, \emph{i.e.}~that $M$ is finite and projective of constant rank as an $A\otimes R$\nobreakdash-module, we will construct an inverse map 
\begin{equation}\label{eq:HartlJuschkaXi}
\Xi_{\kappa}:\Hom_{A\otimes R}\left(\tau^* M, \Hom^{\mathrm{cont}}_{\mathbb{F}}(\KI/A,R)\right)\longrightarrow N,
\end{equation}
to $\tRestau(-\otimes \star)$. The map $\Xi_\kappa$ is inspired by the work of Hartl and Juschka. Concurrently, this proves the compatibility of $\Restau$ with Hartl and Juschka's construction.

\paragraph{Local inverse and Hartl--Juschka's construction.}
The definition of $\Xi_{\kappa}$ builds upon the following lemma:
\begin{Lemma}\label{lem:extension-to-eta-infty}
Given a map $\eta:\tau^*M\to \Hom^{\mathrm{cont}}_{\mathbb{F}}(\KI/A,R)$ of $A\otimes R$-modules, there exists a unique map $\eta_{\infty}$ of $\mathcal{B}_\infty(R)$-modules making the following diagram commute:
\[
\begin{tikzcd}
\tau^* M \arrow[r,"\eta"]\arrow[d] & \Hom^{\mathrm{cont}}_{\mathbb{F}}(\KI/A,R) \arrow[d] \\
\tau^* \mathcal{B}_{\infty}(M) \arrow[r,"\eta_{\infty}"] & \Hom^{\mathrm{cont}}_{\mathbb{F}}(\KI,R)
\end{tikzcd}
\]
where the vertical maps are the canonical ones.
\end{Lemma} 
\begin{proof}
Uniqueness and existence both follow from the universality of the tensor product and the fact that $\Hom^{\mathrm{cont}}_{\mathbb{F}}(\KI,R)$ is canonically a $\mathcal{B}_{\infty}(R)$-module: $\Hom^{\mathrm{cont}}_{\mathbb{F}}(\KI,R)$ is indeed an $\cO_{\infty}\otimes R$-module, where $x=\sum_i{b_i\otimes r_i}$ acts on $h:\KI\to R$ via
\[
x\cdot h:=\left(\KI \ni b\longmapsto \sum_i{h(bb_i)r_i}\right).
\]
Since $h$ is continuous, $(\mathfrak{m}_{\infty}^n\otimes R)\cdot h=0$ for all large enough integers $n$, and thus $\Hom^{\mathrm{cont}}_{\mathbb{F}}(\KI,R)$ is canonically an $\mathcal{A}_{\infty}(R)$-module. It is also a $\KI$-vector space and both actions coincide on $\mathcal{O}_{\infty}$; hence this action extends uniquely to $\mathcal{B}_{\infty}(R)$.
\end{proof}

Let $\eta:M\to \Hom^{\mathrm{cont}}_{\bF}(\KI/A,R)$ be a semilinear map with respect to $\tau$. Lemma \ref{lem:extension-to-eta-infty} guarantees the existence of a unique semilinear extension $\eta_{\infty}:\mathcal{B}_{\infty}(M)\to \Hom^{\mathrm{cont}}_{\bF}(\KI,R)$ which, in addition, is continuous.
\begin{Lemma}\label{lem:etainfty-continuous}
For all $m\in M$, there exists $s_0$ for which $\eta_{\infty}(\sigma^sm)(1)=0$ for all $s\geq s_0$.
\end{Lemma}
\begin{proof}
Since $E$ is abelian, we can pick $(f_1,\ldots,f_\ell)$ a family of generators of $M$ as an $A\otimes R$-module. Because $\eta_{\infty}(f_i)$ is continuous, there exists an integer $N_i$ for which $\eta_{\infty}(f_i)(\mathfrak{m}_{\infty}^n)=0$ for all $n\geq N_i$. Since $E$ is abelian, we may apply Proposition~\ref{prop:existence-of-sigmastablelattice}: for all large enough integer $s$ we have
\[
\sigma^sm\in \mathfrak{m}_{\infty}^{N_1}\mathcal{A}_{\infty}(R)f_1+\ldots+\mathfrak{m}_{\infty}^{N_\ell}\mathcal{A}_{\infty}(R)f_\ell.
\]
For such an $s$ there exist $\alpha_i\in \mathcal{A}_{\infty}(R)$ such that $\sigma^sm=\pi_{\infty}^{N_1}\alpha_1 f_1+\ldots+\pi_{\infty}^{N_\ell}\alpha_\ell f_\ell$, and hence
\begin{align}\label{eq:etainfty}
\eta_{\infty}(\sigma^sm)(1) &= \eta_{\infty}(\pi_{\infty}^{N_1}\alpha_1 f_1+\ldots+\pi_{\infty}^{N_\ell}\alpha_\ell f_\ell)(1) \nonumber \\
&= \sum_{i=1}^\ell{\eta_{\infty}(\pi_{\infty}^{N_i}\alpha_i f_i)(1)} \nonumber \\
&= \sum_{i=1}^\ell{\eta_{\infty}(\alpha_i f_i)(\pi_{\infty}^{N_i})}
\end{align}
Now if $\alpha_i\in \mathcal{A}_{\infty}(R)$ decomposes as $\sum_{j}{b_{ij}\otimes r_{ij}}$ with $b_{ij}\in \mathcal{O}_{\infty}$ and $r_{ij}\in R$, then
\[
\eta_{\infty}(\alpha_i f_i)(\cdot)=\sum_{j}{\eta_{\infty}(f_i)(b_{ij} \cdot )r_{ij}^{q}}
\]
and hence \eqref{eq:etainfty} is zero.
\end{proof}

To find the inverse map, we reduce to the situation where $E$ is split. For this, we let $R\to S$ be a faithfully flat morphism of $A$-algebras on which $E$ becomes isomorphic to $\mathbb{G}_a^d$. Without loss of generality, we may use Lemma \ref{lem:perf-respects-faithfulllyflat} to assume that $S$ is perfect. The pairing then commutes with this particular base change according to Corollary \ref{cor:N-commutes-with-flat-perfect-base-change} and, since faithful flatness preserves isomorphisms, it is enough to prove the perfectness of the pairing after base change to $S$. 

Let $\kappa:E_S\stackrel{\sim}{\to}\mathbb{G}_{a,S}^d$ be a choice of coordinates for $E_S$. For $i\in \{1,\ldots,d\}$, we set $\kappa_i$ (respectively~$\dk_i$) to be the composite maps
\[
\kappa_i:E_S\xrightarrow{\sim} \mathbb{G}_{a,S}^d\xrightarrow{\mathrm{proj}_i} \mathbb{G}_{a,S} \quad \left(\text{respectively}\quad \dk_i:\mathbb{G}_{a,S} \xrightarrow{\mathrm{inj}_i} \mathbb{G}_{a,S}^d \xrightarrow{\sim} E_S \right).
\]
In particular, $\kappa_i\in M_S:=M(E_S)$ and $\dk_i\in N:=N(E_S)$, and we have the relations
\begin{equation}\label{eq:kappa-times-dkappa}
 \kappa_i\circ \dk_j = \delta_{i,j}\in S[\tau]
\end{equation}
for all $1\leq i,j\leq d$, where $\delta_{i,j}=1$ if $i=j$ and $=0$ if $i\ne j$.
Following \cite{hartl-juschka}, to a semilinear map $\eta:M_S\to \Hom^{\mathrm{cont}}_{\bF}(\KI/A,S)$, we assign
\begin{equation}\label{eq:expression-of-xi(eta)}
\Xi_\kappa(\eta):=\sum_{i=1}^d{\dk_i\left\{\sum_{s=0}^{\infty}{\tau^s\eta_{\infty}(\tau^{-(s+1)}\kappa_i)(1)}\right\}} \in N_S
\end{equation}
(it is a formal computation to check that \eqref{eq:expression-of-xi(eta)} coincides with $\check{m}_{\eta}$ in \cite[Theorem 2.5.13]{hartl-juschka} under residue duality). That \eqref{eq:expression-of-xi(eta)} is well-defined amounts to showing that the embraced expression is polynomial in $\tau$; \emph{i.e.} that  $\eta_{\infty}(\tau^{-(s+1)}\kappa_i)(1)=0$ for $s\gg 0$. But this follows from Lemma \ref{lem:etainfty-continuous}.\\

For $n\in N_S$, recall that $\tRestau(-\otimes n)$ defines a semilinear map $M_S\to \Hom^{\mathrm{cont}}_{\bF}(\KI/A,S)$ given by mapping $m$ to the continous $\bF$-linear map $g\mapsto \mathrm{coeff}_0(\tau\circ m\circ \Phi(g)\circ n)$. Theorem \ref{thm:res-perfect} is an immediate consequence of the next proposition.
\begin{Proposition}\label{prop:properties-of-Xi}
The following holds:
\begin{enumerate}[label=$(\roman*)$]
\item\label{item:composition=identity} For all $n\in N_S$, we have $\Xi_{\kappa}(\tRestau(-\otimes n))=n$;
\item\label{item:Xi-injective} $\Xi_\kappa$ is injective.
\end{enumerate}
\end{Proposition}
\begin{proof}
From the definitions, we see that both $\tRestau(-\otimes \star)$ and $\Xi_{\kappa}$ are $S$-linear. So it suffices to show the equality in \ref{item:composition=identity} for $n=\dk_j\tau^l$, $j\in \{1,\ldots d\}$ and $l\geq 0$.
We verify this by direct computation:
\begin{eqnarray*}
\Xi_{\kappa}(\tRestau(-\otimes \dk_j\tau^l)) &=& \sum_{i=1}^d{\dk_i\left\{\sum_{s=0}^{\infty}{\tau^s\mathrm{coeff}_{0}(\tau^{-s}\kappa_i\circ \dk_j \tau^l)}\right\}} \\
&\stackrel{Eq.\eqref{eq:kappa-times-dkappa}}{=} & \dk_j\left\{\sum_{s=0}^{\infty}{\tau^s\mathrm{coeff}_{0}(\tau^{-s} \tau^l)}\right\} 
= \dk_j \tau^l
\end{eqnarray*}

For point \ref{item:Xi-injective}, let $\eta:M_S\to \Hom_{\bF}^{\mathrm{cont}}(\KI/A,S)$ be a semilinear map---equivalently, a linear map $\tau^*M_S\to \Hom_{\bF}^{\mathrm{cont}}(\KI/A,S)$---and assume that $\Xi_{\kappa}(\eta)=0$; \emph{i.e.} that $\eta_{\infty}(\tau^{-(s+1)}\kappa_i)(1)=0$ for all $i\in \{1,\ldots,d\}$ and all $s\geq 0$. We must show that $\eta=0$. \\
We know that $aM_S\ls{\sigma}=M_S\ls{\sigma}$ for all $a\in \KI\setminus\{0\}$. Hence, for all $m\in M_S$ and all $a\in \KI$, we may write $am$ as
\[
am=\sum_{i=1}^d{f_i(\sigma)\kappa_i} \quad \text{where~}f_i(\sigma)\in S\ls{\sigma}.
\]
For all $i$, we decompose $f_i(\sigma)$ uniquely as $f_i^{-}(\sigma)+f_i^{+}(\tau)$ where $f_i^{-}(\sigma)\in \sigma S\ps{\sigma}$ and $f_i^{+}(\tau)\in S[\tau]$. We then have
\begin{align*}
\eta(m)(a)&=\eta_{\infty}(m)(a)=\eta_{\infty}(am)(1)=\eta_{\infty}\left(\sum_{i=1}^d{f_i(\sigma)\kappa_i}\right)(1)\\
&=\eta_{\infty}\left(\sum_{i=1}^d{f_i^{-}(\sigma)\kappa_i}\right)(1)+\eta_{\infty}\left(\sum_{i=1}^d{f_i^{+}(\tau)\kappa_i}\right)(1)
\end{align*}
But both are zero: the former because, by assumption, $\eta_{\infty}(\tau^{-(s+1)}\kappa_i)(1)=0$ for all $i$ and $s\geq 0$ and  because $\eta_{\infty}$ is continuous for the $\infty$-adic topology; the latter because the expression in parenthesis belongs to $M_S$ and $1\in A$. 
\end{proof}

\begin{proof}[Proof of Theorem \ref{thm:res-perfect}]
As it is enough to prove the perfectness of the pairing after a faithfully flat base change $R\to S$, we may assume that there exists a choice of coordinates $\kappa$ for $E_S$.
First consider the case that $E$ is abelian (hence $\Xi_{\kappa}$ is defined). Proposition \ref{prop:properties-of-Xi} then implies that the following diagram commutes:
\[
\begin{tikzcd}[column sep=5em]
N_S \arrow[r,"{\tRestau(-\otimes \star)}"] \arrow[dr,"\id_{N_S}"'] & {\Hom_{A\otimes S}\left(\tau^*M_S,\Hom_{\mathbb{F}}^{\mathrm{cont}}(\KI/A,S)\right)} \arrow[d,"{\Xi_{\kappa}}"] \\
& N_S
\end{tikzcd}
\]
In particular $\Xi_\kappa$ is surjective; it is also injective by Proposition \ref{prop:properties-of-Xi}, hence is bijective and so is $\tRestau(-\otimes \star)$. As mentioned before, this implies that $\Restau(-\otimes \star)$ is bijective. 

Since by assumption $M$ (and hence also $\tau^*M$) is finite projective as $A\otimes R$-module, the isomorphism $\Restau(-\otimes \star)$ identifies $N$ with $(\tau^*M)^\vee\otimes_{A\otimes R} (\ka \otimes R)$, where $(\tau^*M)^\vee$ denotes the dual module. 
 In particular, $N$ is also a finite projective $A\otimes R$-module of the same rank. Further via this identification, $\Restau$ is just the composition
\[\tau^*M\otimes_{A\otimes R} N \xrightarrow{\sim} \tau^*M\otimes_{A\otimes R} (\tau^*M)^\vee\otimes_{A\otimes R} (\ka \otimes R) \xrightarrow{\mathrm{ev}\otimes \id} (A\otimes R) \otimes_{A\otimes R} (\ka \otimes R)=\ka \otimes R.\]
This shows that the map $\Restau(\star\otimes -)$ is bijective, too.

\medskip

The case that $E$ is coabelian runs parallel to the abelian case by switching the roles of $N$ and $\tau^*M$.
Instead of the map $\Xi_\kappa$, we construct an inverse to (the base extension to $S$ of) $\tRestau(\star \otimes -)$:
\[  \check{\Xi}_\kappa : \Hom_{A\otimes R}(N_S, \Hom_{\bF}^{\mathrm{cont}}(\KI/A,S)) \to \tau^*M_S \]
given by
\begin{equation}
 \check{\Xi}_\kappa(\check{\eta}) := \sum_{i=1}^d\left( \sum_{s=0}^\infty \check{\eta}_\infty(\dk_i\tau^{-(s+1)})(1)\otimes \tau^{s} \kappa_i \right) \in \tau^*M_S
\end{equation}
where $\check{\eta}_\infty: \mathcal{B}_{\infty}(N_S)\to \Hom_{\bF}^{\mathrm{cont}}(\KI,S)$ is the unique continuous extension of the map $\check{\eta}:N_S\to \Hom_{\bF}^{\mathrm{cont}}(\KI/A,S)$ (guaranteed by Lemma \ref{lem:extension-to-eta-infty} for $N_S$ in place of $\tau^*M$), and where the sum is always finite by the analogous statement to Lemma \ref{lem:etainfty-continuous} with right $R\ls{\sigma}$-modules instead of left $R\ls{\sigma}$-modules.

Along the same lines as in the proof of Proposition \ref{prop:properties-of-Xi}, one verifies that $\check{\Xi}_\kappa$ is indeed the inverse of $\tRestau(\star \otimes -)$. Hence $\tRestau(\star \otimes -)$ is an isomorphism, and it identifies $\tau^*M$ with $N^\vee\otimes_{A\otimes R} (\ka \otimes R)$ which shows that $\tau^*M$ is also finite projective over $A\otimes R$, and that $\Restau(-\otimes \star)$ is bijective, too.
\end{proof}

\section{Applications}\label{sec:applications}
As before, let $E$ be an Anderson $A$-module over $R$. We let $M=M(E)$ be the motive of $E$ and $N=N(E)$ be its comotive.
In this section, we do not assume that $R$ is perfect unless explicitly stated. 

\subsection{Abelian equals coabelian}
The main application of Theorem \ref{thm:res-perfect} is the ``abelian equals coabelian'' statement.
\begin{Corollary}\label{cor:abelian=coabelian}
Assume that $R$ is perfect. Then $E$ is abelian if, and only if $E$ is coabelian. 
\end{Corollary}

It would be satsifying to know whether some version of Corollary \ref{cor:abelian=coabelian} extend beyond the perfect case\footnote{The fact that $N$ does not commute with base-change prevent the ``naive'' descent argument along $R\to R^{\mathrm{perf}}$ to work. A comparable mistake was made in a previous version of this text.}. Note that the techniques employed here  are probably useless as the reside-in-$\tau$ pairing does not exist in the imperfect situation. Indeed, if $E$ is an Anderson $A$-module over $R$, where $R$ is not necessarily perfect, then the pairing exists over its perfection:
\[
\tau^*M \otimes_{A\otimes R} N\longrightarrow \ka\otimes R^{\mathrm{perf}}.
\]
Even for Drinfeld modules, this map does not factor through $\ka\otimes R$ as the examples below shows (see Section \ref{sec:computations}). However, if $E$ is abelian and coabelian, then the source is a finite $A\otimes R$-module and hence the map factors through some $R^{1/q^b} \subset R^{\mathrm{perf}}$ for some $b\geq 1$. We expect that the smallest such $b$, which we denote by $b(E)$, may be chosen to be not too large compared to $E$. The following conjecture was supported by several examples. 

\begin{Conjecture}
Assume that $R$ is a field. The constant $b(E)$ defined above satisfies $|b(E)w(E)|<1$ for all weights $w$ of $E$.
\end{Conjecture}

\begin{Remark}
We refer the reader to \emph{e.g.}, \cite[\S 5.2]{taelman-artin} or \cite[Definition 3.20]{gazda} for an account on the theory of \emph{weights} of Anderson modules and motives. We use absolute values in the statement, although $b(E)$ is always positive, because the sign of the weights depends on the reference. If one follows \cite{gazda}, then one may remove absolute values.
\end{Remark}

\subsection{Tensor equivalence}\label{subsec:tensor}
Let $E$ and $E'$ be two abelian Anderson $A$-modules over a perfect $A$-algebra $R$. The tensor product $M(E)\otimes M(E')$ in the category of $A$-motives over $R$ is well-defined and we say \emph{$E$ and $E'$ admit a tensor product} if $M(E)\otimes M(E')$ is in the essential image of the functor
\[
M:\mathbf{And}_R\longrightarrow A\mathbf{Mot}_R
\]
from the category of Anderson $A$-modules over $R$ to that of $A$-motives over $R$ \cite[Theorem 3.5]{hartl}. Note that, contrary to the case where $R$ is a field, it is unclear to us whether the tensor product of $E$ and $E'$ always exists. If it does, then all Anderson $A$-modules $E''$ satisfying $M(E'')\cong M(E)\otimes M(E')$ are isomorphic; we then call $E''$ \emph{a tensor product of $E$ and $E'$} and denote it by $E\otimes E'$.\\

We call $E\otimes^{\co} E'$, when it exists, the Anderson module obtained by a similar strategy, replacing the functor $M$ with $N$.
\begin{Corollary}
$E\otimes^{\mathrm{co}}E'$ is isogenous to $E\otimes E'$ when they exist.
\end{Corollary}
\begin{proof}
We have the following sequence of isomorphisms of left $A\otimes R[\tau]$-modules 
\begin{align*}
\tau^*M(E\otimes E') &= \tau^*M(E)\otimes_{A\otimes R} \tau^*M(E') \\
&=\Hom(N(E),\ka\otimes R)\otimes_{A\otimes R} \Hom(N(E'),\ka\otimes R) \\
&= \Hom(N(E)\otimes_{A\otimes R} N(E'),\ka\otimes R)\otimes_A \ka \\
&= \Hom(N(E\otimes^{\co} E'),\ka\otimes R)\otimes_A \ka \\
&= \tau^*M(E\otimes^{\co} E')\otimes_A \ka.
\end{align*}
For the former equality we used that pullbacks commute with tensor product, for the second we used Theorem \ref{thm:res-perfect}, for the third that all modules involved are finite projective, for the fourth the definition of $\otimes^{\mathrm{co}}$ and for the fifth we used Theorem \ref{thm:res-perfect} again.

Because $R$ is perfect, pullback by $\sigma$ yields $M(E\otimes E')\cong M(E\otimes^{\co} E')\otimes_A \ka$. This implies that $M(E\otimes E')$ and $M(E\otimes^{\co} E')$ are isogenous and we conclude using Theorems 3.5 and 5.9 in~\cite{hartl}.
\end{proof}

\subsection{Anderson generating functions and special functions}\label{subsec:anderson-generating-functions}
In this subsection, we explain how the residue-in-$\tau$ pairing allows to compute special functions and Anderson generating functions. For doing so, we have to review some notion from \cite[Section 2]{gazda-maurischat} and \cite[Section 5]{hartl-juschka}.

Let $R=\CI$ be the completion of an algebraic closure of $K=\Frac(A)$ at the place $\infty$, and let $\T=A\hat{\otimes}\CI$ denote the completion of $A\otimes \CI$ with respect to the Gauss norm (we refer to \cite[Section 2]{gazda-maurischat} for details). We further let $H_B(E)$ denote \emph{the period lattice} of $E$, \emph{i.e.}, the kernel of the exponential map $\exp_E:\mathrm{Lie}_E(\CI)\to E(\CI)$. Let also $E(\T):=A\hat{\otimes} E(\CI)$ be the Gauss completion of $A\otimes E(\CI)$ with respect to the norm on $E(\CI)$ induced after a choice of coordinates $\kappa:E\xrightarrow{\sim}\Ga^d$. Recognize that $E(\T)$ carries two distinct $A$-actions,  one from the multiplication on the left factor, the other by the $A$-action on $E(\CI)$, both extended by continuity.

In \cite[Definition 3.1]{gazda-maurischat}, the module of \emph{special functions} was defined as
 \[  \sf(E):= \{ \omega \in E(\T) \mid \text{for all } a\in A: (a\otimes 1)\omega =(1\otimes a)\omega \},\]
and we established an isomorphism 
\begin{equation}\label{eq:delta_E}
 \delta_E: H_B(E) \stackrel{\sim}{\longrightarrow} \sf(E)\otimes_A \ka
\end{equation}
(see \cite[Theorem 3.11 and Remark 3.14]{gazda-maurischat}). In the case of $A=\FF[t]$, this isomorphism precisely recovers the map which sends a period $\lambda\in H_B(E)$ to its \emph{Anderson generating function}
\[  \cG_\lambda(t) = \sum_{n=0}^\infty t^n\otimes \exp_E( (\mathrm{Lie}_E(t))^{-n-1}\lambda) \in E(\T) \]
(or, more precisely, to $\cG_\lambda(t) dt$; see \cite[Proposition 3.13]{gazda-maurischat}).

\medskip

We now assume that the Anderson $A$-module $E$ is uniformizable and abelian, and we let $r$ denote the common rank of $M=M(E)$ and $N=N(E)$ as $A\otimes \CI$-modules.
In this case, the period lattice $H_B(E)$ is a locally free $A$-module of the same rank $r$, as are the $A$-modules
\[ 
H_B(M) =  \{ m\in M\otimes_{A\otimes\CI} \T \mid \tau \cdot m=m \}\quad \text{and}\quad H_B(N) =  \{ n\in N\otimes_{A\otimes\CI} \T \mid n\cdot \tau=n \}.
\]
In addition, there is an isomorphism of $A$-modules 
\begin{equation}\label{eq:sf-to-cohomology-of-M}
\gamma: \sf(E)\stackrel{\sim}{\longrightarrow} \Hom_A(H_B(M),A),\,\, \omega\longmapsto g_\omega,
\end{equation} 
induced by the map $E(\CI)\to \Hom_{\CI}(M,\CI), e\mapsto [ g_e: m\mapsto m(e) ]$ (see \cite[Theorem 3.3]{gazda-maurischat}).

The composite isomorphism 
\[\beta:=\gamma\circ \delta_E :\quad H_B(E) \stackrel{\sim}{\longrightarrow} \Hom_A(H_B(M),A)\otimes_A \ka=\Hom_A(H_B(M),\ka),\]
was already established by Anderson in \cite[Corollary 2.12.1]{anderson} (see also \cite[Remark 2.5.34]{hartl-juschka}), and Hartl--Juschka further showed that it fits into a commutative diagram of (isomorphisms of) $A$-modules (see \cite[Proposition 2.5.48]{hartl-juschka})
\[
\begin{tikzcd}[column sep=6em]
\Hom_{A\otimes\CI[\tau]}(\tau^*M, \ka\otimes_A \T) \arrow[d,"h_M^\vee"'] \arrow[r,"H_B(\Xi)"]&  H_B(N)  \arrow[d,"H_B(\eps_0)"'] \\
 \Hom_{A}(H_B(M), \ka)&  H_B(E)\arrow[l,"\beta"'] 
\end{tikzcd}
\]
Here, $H_B(\Xi)$ is obtained by extending $\Xi: \Hom_{A\otimes\CI}(\tau^*M, \ka\otimes \CI)\to N\ $ along $A\otimes\CI\to \T$ and then restricting to the respective $\tau$-invariant subspaces. The map $h_M:H_B(M)\otimes_A \T\xrightarrow{\sim} M\otimes_{A\otimes \CI} \T$ is the $\T$-linear isomorphism ensured by uniformizablity, and $H_B(\eps_0)$ is given by restricting the $\T$-linear extension of the map $\eps_0:N=\Hom(\Ga,E)\to \mathrm{Lie}_E(\CI), n\mapsto \mathrm{Lie}(n)(1)$.

Because the residue-in-$\tau$ pairing provides an inverse to $\Xi$, we may immediately translate Hartl--Juschka's results into the following statement:

\begin{Theorem}
If $E$ is a uniformizable abelian $A$-module over $\CI$, the map $\delta_E:H_B(E)\to \sf(E)\otimes_A \ka$ given in Formula \eqref{eq:delta_E} sits in the commutative diagram
\[
\begin{tikzcd}[column sep=10em]
 H_B(N) \arrow[r, "H_B(\Restau(-\otimes \star))"]\arrow[dd,"H_B(\eps_0)"] &  \Hom_{A\otimes\CI[\tau]}(\tau^*M, \ka\otimes_A \T) \arrow[d,"h_M^\vee"']\\ 
   &  \Hom_{A}(H_B(M), \ka) \\
   H_B(E) \arrow[r, "\delta_E"] & \sf(E)\otimes_A \ka   \arrow[u,"\gamma"]
\end{tikzcd}
\]
\end{Theorem}
 
The explicit descriptions of these maps are also crucial for Namoijam and Papanikolas investigations on Anderson's exponentiation theorem and descriptions of quasi-periods and quasi-logarithms in terms of rigid analytic trivializations (see \cite[Section 4.4]{NP}). 
Their main reason for restricting to so-called \emph{almost strictly pure} $t$-modules in \cite[Section 4.5]{NP} is that they found an explicit description of the map $\Xi=\Restau(-\otimes \star)^{-1}$ for those $t$-modules.

So our residue map may enable one to generalize their formulas to arbitrary uniformizable abelian $t$-modules.

\subsection{Barsotti-Weil type formula}\label{subsec:barsotti-weil}
In this section, we explain how to recover a formula due to Taelman \cite[Theorem 8.1.1]{taelman}, generalizing a formula of Papanikolas and Ramachandran \cite{PR}, and Woo \cite{woo}. We also refer the reader to the recent work of 
Głoch--K\c{e}dzierski--Krasoń \cite{gkk,kk,kp} for further results in this direction. The formula was subsequently generalized for general coefficients by Mornev \cite[Theorem 9.4]{mornev}. 

\begin{Theorem}\label{thm:barsotti-weil}
Let $E$ be an abelian Anderson module over a perfect $A$-algebra $R$. Then, there is an isomorphism of $A$-modules
\[
E(R)\cong \mathrm{Ext}^{1}_{A\mathbf{Mot}_R}(M(E),\mathbbm{1})\otimes_A \ka
\]
where the extension module is taken in the category of effective $A$-motives over $R$. 
\end{Theorem}
\begin{proof}
By definition of morphisms in the category of $A$-motives over $R$, the inclusion of effective $A$-motives in the category of left $(A\otimes R)[\tau]$\nobreakdash-modules is full and faithful. Consequently, extensions of effective $A$-motives can be computed there. Because $M=M(E)$ is finite projective over $A\otimes R$, so is $\Theta(M):=\bigoplus_{i\geq 0}{\tau^{i*}M}$ over $(A\otimes R)[\tau]$ where $\tau$ now acts as $\tau_{\Theta}:(m_0,m_1,m_2,\ldots)\mapsto (0,\tau^* m_0,\tau^*m_1,\ldots)$. In addition, we have a map $\theta:\Theta(M)\to M$, $(m_i)_i\mapsto \sum_{i}{\tau^i_M(m_i)}$.  It is easy to see that $0\to \tau^*\Theta(M) \xrightarrow{\tau_\Theta-\tau_M} \Theta(M)\xrightarrow{\theta} M\to 0$ is a projective resolution of $M$ in the category of $(A\otimes R)[\tau]$-modules, and applying $\Hom_{(A\otimes R)[\tau]}(-,\mathbbm{1})$ to it leads to a long exact sequence of $A$-modules
\begin{equation}\label{eq:les}
0\to \Hom_{A\mathbf{Mot}_R}(M,\mathbbm{1})\to \Hom_{A\otimes R}(M,A\otimes R) \xrightarrow{\tau-\tau_M^{\vee}} \Hom_{A\otimes R}(\tau^* M,A\otimes R) \to \mathrm{Ext}^1_{A\mathbf{Mot}_R}(M,\mathbbm{1})\to 0
\end{equation}
where the middle arrow acts as $f\mapsto \tau\circ f-f\circ \tau_M$. Note that in the category of $A$-modules, $\tau^*H$ and $H$ are identified for all $A\otimes R$-modules $H$ (the $R$-module structure is modified but not the $A$-module one). As such, we have
\begin{align*}
\Hom_{A\otimes R}(M,A\otimes R) &= \tau^*\Hom_{A\otimes R}(M,A\otimes R) \\
&\cong \Hom_{\tau^* A\otimes R}(\tau^* M,\tau^* A\otimes R) \\
&\cong \Hom_{A\otimes R}(\tau^*M,A\otimes R)
\end{align*}
in the category of $A$-modules. Consequently, we voluntarily forget $\tau$-pullbacks in homomorphisms when considered as $A$-modules. 

On the other-hand, by Proposition \ref{prop:commutation-with-tau} we have a commutative diagram
\[
\begin{tikzcd}[column sep=3em]
\Hom_{A\otimes R}(M, \ka\otimes R) \arrow[r,"\tau-\tau_M^{\vee}"]& \Hom_{A\otimes R}(M, \ka\otimes R) \\
N \arrow[u,"{n\mapsto \mathrm{Res}_{\tau}(-\otimes n)}"]\arrow[r,"\id-\tau"] & N\arrow[u,"{n\mapsto \mathrm{Res}_{\tau}(-\otimes n)}"']
\end{tikzcd}
\]
whose top row identifies with the middle arrow of \eqref{eq:les} in the category of $A$-modules. In addition, the vertical maps are isomorphisms by Theorem \ref{thm:res-perfect}. Altogether, the sequence \eqref{eq:les} becomes:
\[
0\to \Hom_{A\mathbf{Mot}_R}(M,\mathbbm{1})\otimes_A \ka \to N(E) \xrightarrow{\id-\tau} N(E) \to \mathrm{Ext}^1_{A\mathbf{Mot}_R}(M,\mathbbm{1})\otimes \ka \to 0
\]
and we conclude using Lemma \ref{lem:anderson-exact-sequence}.
\end{proof}

The proof of Theorem \ref{thm:barsotti-weil} also implies the following vanishing result which, as the referee kindly pointed out, is usually proven over fields by invoking the faithful inclusion of $A$-motives into isocrystals \cite[Definition~3.16]{gazda} and then appealing to \cite[Theorem~2.3.14(a)]{hartl-juschka} and \cite[Proposition~3.4.12]{mornev}.
\begin{Corollary}
We have $\Hom_{A\mathbf{Mot}_R}(M(E),\mathbbm{1})=(0)$. 
\end{Corollary}

\begin{Remark}
We believe that the assumption that $R$ is perfect is not needed for the statement of Theorem \ref{thm:barsotti-weil}. But the proof would then require different techniques, closer to Mornev's approach in \cite[\S 8.9]{mornev-shtuka}.
\end{Remark}

\subsection{Twisted $L$-series of Anderson modules}
In \cite{angles-tavares-ribeiro}, Anglès--Ngo Dac--Tavares Ribeiro introduced a deformation of (models of) Drinfeld modules where a transcendental variable $T$ ($z$ in their notation) appears. Their definition depends \emph{a priori} on choices of coordinates. Using the residue-pairing, we show that it does not depend on such a choice by providing an alternative construction of the $T$-deformation. \\

Let $E$ be an abelian Anderson $A$-module over a perfect $A$-algebra $R$. Let $N(E)=\Hom(\mathbb{G}_a,E)$ denote its $A$-comotive.
\begin{Definition}
We define the \emph{$T$-deformation of $E$} to be the functor
\[
E_T: \mathbf{Alg}_R\longrightarrow \mathbf{Mod}_{A[T]}, \quad S\longmapsto N(E_S)
\]
where $N(E_S)$ is seen as an $A[T]$-module where $T$ acts as $\tau$.
\end{Definition}
\begin{Remark}
The naming \emph{$T$-deformation} is understood as follows. From the exact sequence $0\to N(E)\xrightarrow{\id-\tau} N(E)\to E(R)\to 0$ of Lemma \ref{lem:anderson-exact-sequence}, we have a commutative diagram
\[
\begin{tikzcd}
\mathbf{Alg}_R \arrow[r,"E_T"]\arrow[dr,"E"'] & \mathbf{Mod}_{A[T]} \arrow[d,"T=1"]\\
& \mathbf{Mod}_{A}
\end{tikzcd}
\]
In particular, the data of $E_T$ recovers that of $E$ at $T=1$.
\end{Remark}

If $R=k$ is a finite field, then $E_T(k)$ is a finite $A[T]$-module. In particular we may consider its Fitting ideal. The next proposition proves the compatibility of $L$-series as considered in \cite{angles-tavares-ribeiro} and the ones defined in \cite{caruso-gazda}.
\begin{Proposition}
We have $\mathrm{Fitt}~E_T(k)=\det_{A[T]}(T-\tau~|~M(E)\otimes_k k[T])$.
\end{Proposition}
\begin{proof}
We have an exact sequence $0\to N(E)\otimes_k k[T]\xrightarrow{T-\tau} N(E)\otimes_k k[T]\to E_T(k) \to 0$ which lifts that of Lemma \ref{lem:anderson-exact-sequence}: it is obtained by applying the functor $N(E)\otimes_{k[\tau]}-$ to the exact sequence $0\to k[\tau,T]\xrightarrow{T-\tau} k[\tau,T]\to k[\tau]\to 0$ of left $k[\tau]$-modules, and noticing that $T-\tau$ remains injective as one shows using the grading of $N(E)\otimes_k k[T]$ by the powers of $T$.

Let $d$ be the dimension of $k$ over $\mathbb{F}$. We deduce
\begin{align*}
\mathrm{Fitt}~E_T(k)&=\mathrm{det}_{A[T]}(T-\tau~|~N(E)\otimes_k k[T]) \\
&=\mathrm{det}_{A\otimes k[T]}(T^d-\tau^d~|~N(E)\otimes_k k[T]) \\
&=\mathrm{det}_{A\otimes k[T]}(T^d-\tau^d~|~\Hom_{A\otimes k}(\tau^*M(E),\ka\otimes_k k[T]) \\
&=\mathrm{det}_{A\otimes k[T]}(T^d-\tau^d~|~\tau^*M(E)\otimes_k k[T]) \\
&=\mathrm{det}_{A[T]}(T-\tau~|~\tau^*M(E)\otimes_k k[T]) \\
&=\mathrm{det}_{A[T]}(T-\tau~|~M(E)\otimes_k k[T]).
\end{align*}
The second and fifth equalities follow from \cite[Lemma 8.1.4]{boeckle-pink}, the third from the residue-in-$\tau$ pairing. The fourth equality is the compatibility among determinants and duals, and the last from the fact that $M(E)=\tau^*M(E)$ as $A$-modules.
\end{proof}

\section{Computations}\label{sec:computations}

In this final section, we illustrate the residue-in-$\tau$ pairing by computing some examples.

Throughout, the curve $C$ is the projective line $\mathbb{P}^1$ over $\mathbb{F}$ and $\infty$ is the ``north pole'' $[0:1]$. Thus $A$ identifies with the polynomial ring $\mathbb{F}[t]$ where $t$ is any function on $\mathbb{P}^1$ with a simple pole at $\infty$ and regular elsewhere. The field $K_{\infty}$ identifies with $\mathbb{F}\ls{\varpi}$ where $\varpi:=1/t$.

We let $R$ be a perfect $\mathbb{F}[t]$-algebra and denote by $\theta$ the image of $t$ in $R$.
For what follows, we identify the ring $A\otimes R$ as $R[t]$.

We recall from Remark \ref{rem:explicit-residue-duality} that in this setting for an abelian Anderson $t$-module $(E,\Phi)$, the residue-in-$\tau$ pairing is given by
\begin{align}
\Restau:\tau^*M(E)\otimes_{R[t]}N(E) &\longrightarrow R[t]dt,\notag \\
 m\otimes n \hspace*{1.15cm} &\longmapsto -\sum_{k=1}^\infty \mathrm{coeff}_{0}\left( \tau\circ m\circ \Phi(\varpi^{k})\circ n\right) t^{k-1} dt. \label{eq:res-pairing}
\end{align}

\subsection{Drinfeld modules}
Let $E$ be a Drinfeld module of rank $r$ over $R$ \cite[Definition 3.7]{hartl}. By working Zariski locally on $\Spec R$, we may assume that $E$ is equal to $\mathbb{G}_a$ so that the $\mathbb{F}[t]$-module structure on $E$ amounts to a ring homomorphism $\Phi:\mathbb{F}[t]\to R[\tau]$ of the form
\[
\Phi(t):=\theta+g_1\tau+\ldots+g_r\tau^r
\]
where $g_i\in R$ and $g_r\in R^{\times}$. It will be convenient to extend the notation $g_i$ for $i\in \mathbb{N}$ by declaring $g_0:=\theta$ and $g_i=0$ for $i>r$.\\ 

It is well-known that the motive and the comotive of $E$ are free of rank $r$ over $R[t]$ with basis given by $(1,\tau,\ldots,\tau^{r-1})$. In particular $E$ is both abelian and coabelian.

The computation of the pairing was partially done in \cite[Example 2.5.16]{hartl-juschka}. But we will see that our approach admits a much shorter calculation.

By $R[t]$-linearity, it suffices to do the computation of the differentials $\mathrm{Res}_{\tau}(\tau^i\otimes \tau^j)$ for all integers $0\leq i,j<r$. In this direction, we prove:
\begin{Proposition}\label{prop:computation-drinfeld}
The map \eqref{eq:res-pairing} sends $\tau^i\otimes \tau^j$ to the differential form
\[
\left( -\delta_{i+j,r-1}  +\!\!\!\! \sum_{\substack{n\in \mathbb{N}_{>0},~v_i\in \{1,\ldots,r\} \\ v_1+\ldots+v_n=1+i+j-r}} \!\!\!\!\!(-1)^{n+1}\!\left(\frac{g_{r-v_1}}{g_r}\right)^{q^{v_1+\ldots+v_n-j}}\!\!\!\left(\frac{g_{r-v_2}}{g_r}\right)^{q^{v_2+\ldots+v_n-j}}\!\!\!\cdots\!\left(\frac{g_{r-v_n}}{g_r}\right)^{q^{v_n-j}}\right)\!\frac{dt}{g_r^{q^{-j}}}
\]
where, as usual, $\delta_{i+j,r-1}=1$ if $i+j=r-1$ and $=0$ if $i+j\ne r-1$.
\end{Proposition}
\begin{proof}

By writing 
\[ \Phi(t)=g_r(\frac{g_0}{g_r}\sigma^r+\frac{g_1}{g_r}\sigma^{r-1}+\ldots+\frac{g_{r-1}}{g_r}\sigma+1)\tau^r,\]
 we see that 
 \[ \Phi(t)^{-1}=  \sigma^r \left(1+\frac{g_{r-1}}{g_r}\sigma+\ldots+\frac{g_1}{g_r}\sigma^{r-1}+\frac{g_0}{g_r}\sigma^r\right)^{-1} \frac{1}{g_r} \in \sigma^rR\ps{\sigma}. \] In particular, we deduce that the zeroth coefficient of $\tau^{i+1}\Phi(\varpi^k)\tau^j=\tau^{i+1}\Phi(t)^{-k}\tau^j$ is zero whenever $i+j+1<kr$. Since we chose $i,j\leq r-1$, this shows that this coefficient is $0$ if $k\geq 2$.

Therefore, the computation of the pairing reduces to that of the zeroth coefficient of $\tau^{i+1}\Phi(t)^{-1}\tau^j$. We have:
\begin{align*}
\Phi(t)^{-1} &=\sigma^r \left(1+\frac{g_{r-1}}{g_r}\sigma+\ldots+\frac{g_1}{g_r}\sigma^{r-1}+\frac{g_0}{g_r}\sigma^r\right)^{-1} \frac{1}{g_r} \\
&=\sigma^r\left[\sum_{n=0}^{\infty}{(-1)^n\left(\frac{g_{r-1}}{g_r}\sigma+\ldots+\frac{g_1}{g_r}\sigma^{r-1}+\frac{g_0}{g_r}\sigma^r\right)^n}\right]\frac{1}{g_r}
\end{align*}
The summand for positive $n$ can be expanded as follows:
\[
\left(\frac{g_{r-1}}{g_r}\sigma+\ldots+\frac{g_1}{g_r}\sigma^{r-1}+\frac{g_0}{g_r}\sigma^r\right)^n =\sum_{1\leq v_1,\ldots, v_n\leq r}{\left(\frac{g_{r-v_1}}{g_r}\right)\sigma^{v_1}\left(\frac{g_{r-v_2}}{g_r}\right)\sigma^{v_2}\cdots \left(\frac{g_{r-v_n}}{g_r}\right)\sigma^{v_n}}.
\]
Commuting all $\sigma^{v_i}$ to the left, altogether we obtain
\[
\Phi(t)^{-1}=\sigma^r\frac{1}{g_r}+\sum_{m>0}{\sigma^{r+m}\!\!\!\!\!\sum_{\substack{n\in \mathbb{N}_{>0},~v_i\in \{1,\ldots,r\} \\ v_1+\ldots+v_n=m}}{\!(-1)^n\!\left(\frac{g_{r-v_1}}{g_r}\right)^{q^{v_1+\ldots+v_n}}\!\!\!\left(\frac{g_{r-v_2}}{g_r}\right)^{q^{v_2+\ldots+v_n}}\!\!\!\cdots\!\left(\frac{g_{r-v_n}}{g_r}\right)^{q^{v_n}}\!\frac{1}{g_r}}}.
\]
This formula suffices to conclude. 
\end{proof}

\begin{Remark}
We thank Ferraro for pointing out that Proposition \ref{prop:computation-drinfeld} could also be recovered as the combination of his Theorems 7.18 and 7.38 in \cite{ferraro}.
\end{Remark}

\subsection{Tensor powers of the Carlitz module}
Recall that the Carlitz module $\carl$ is the Drinfeld module of rank one which is equal to $\mathbb{G}_a$ and whose action of $t$ is described by $\Phi(t)=\theta+\tau$. Both its motive and comotive are free of rank one over $R$, generated by $\id_{\mathbb{G}_a}$. According to Proposition \ref{prop:computation-drinfeld}, we have 
\[
\mathrm{Res}_{\tau}(\id_{\mathbb{G}_a}\otimes \id_{\mathbb{G}_a})=-dt.
\]
We now undertake the computation of the $d$th tensor power of the Carlitz module, where $d$ is a positive integer, and where the tensor product is either $\otimes$ or $\otimes^{\co}$ of Subsection \ref{subsec:tensor} (because $\ka\cong \mathbb{F}[t]$, the operations $\otimes$ and $\otimes^{\co}$ are equivalent). 

Recall from e.g. \cite[Section 1.5.3]{rapid-intro} 
that the \emph{$d$-th tensor power $\carl^{\otimes d}$ of the Carlitz module} is the Anderson module 
which, as an $\mathbb{F}$-vector space scheme is $\mathbb{G}_a^d$, and whose $t$-action corresponds in canonical coordinates to the matrix
\[ \Phi(t) =\begin{pmatrix} 
\theta & 1 & 0 & \cdots & 0\\ 
0 & \ddots & \ddots & \ddots & \vdots\\
\vdots & \ddots &  \ddots & \ddots & 0 \\
0 & & \ddots & \ddots & 1 \\
\tau & 0 & \cdots & 0 & \theta \end{pmatrix} \in \Mat_{d\times d}(R\sp{\tau}).
\]
The motive $M(\carl^{\otimes d})$ is a free $R[t]$-module of rank $1$ with basis $\{\kappa_1:\mathbb{G}_a^d\to \mathbb{G}_a\}$, where $\kappa_1$ is the projection to the first coordinate, and the comotive $N(\carl^{\otimes d})$ is a free $R[t]$-module of rank $1$ with basis $\{\dk_d:\mathbb{G}_a\to \mathbb{G}_a^d\}$, where $\dk_d$ corresponds to the $d$th coordinate. Hence computing the pairing, due to $R[t]$-sesquilinearity, amounts to consider $\kappa_1\circ \Phi(f) \circ \dk_d$ for $f=\varpi^k=t^{-k}\in \KI$ (with $k\geq 1$). In matrix view, this is just the entry in the upper right corner of the matrix $\Phi(f)$ and $\coeff_{0}( \tau\circ \kappa_1\circ \Phi(f)\circ \dk_d)$ is just the right coefficient in $\sigma$ of this entry.

The matrix $\Phi(\varpi)=\Phi(t)^{-1}$ can be computed using the Gaussian algorithms, taking into account that row operations correspond to multiplications with matrices from the left, \emph{i.e.}, that we have to multiply scalars in the elementary operations from the left. This results in $\Phi(\varpi)=C_0+\sigma D$ where
\begin{align*}
 C_0 &=\begin{pmatrix} 
 0 & \cdots & \cdots & \cdots & 0\\ 
1 & \ddots &  &  & \vdots\\
-\theta & \ddots & \ddots & & \vdots \\
\vdots & \ddots & \ddots & \ddots & \vdots \\
(-\theta)^{d-2} & \cdots & -\theta & 1 & 0 \end{pmatrix}\in  \Mat_{d\times d}(R),\text{ and } \\[2mm]
D &= \left(  (-\theta)^{q(i-1)}  \left(1-\sigma (-\theta)^{qd}\right)^{-1} (-\theta)^{d-j}   \right)_{i,j}\in  \Mat_{d\times d}(R\sps{\sigma}).
\end{align*}
The coefficient of $\sigma$ of the upper right entry of $\Phi(\varpi)=C_0+\sigma D$ is therefore the coefficient of $\sigma^0$ of the upper right entry of $D$, \emph{i.e.}, equals $1$.\\
For $k\geq 2$, we have
\[ \Phi(\varpi^{k})=(C_0+\sigma D)^k\equiv C_0^k + C_0^{k-1}\sigma D+C_0^{k-2}\sigma D C_0+\cdots + \sigma D C_0^{k-1} \mod
 \sigma^2. \]
Since the first row and the last column of $C_0$ are zero, the upper right entries of all the products on the right hand side are $0$. Hence, $\coeff_{0}(\tau\circ \kappa_1\circ \Phi(\varpi^k) \circ\dk_d)=0$ as soon as $k\geq 2$.

In total, we obtain
\[ \Restau(\kappa_1\otimes \dk_d) = -dt. \]

\subsection{Maurischat's example}

In \cite[Example 6.3]{maurischat}, the second author provided an example of a simple Anderson $t$-module which nevertheless is not pure (unless the characteristic of $\bF$ is $2$).
It is defined over the rational function field $\Fq(\theta)$, and is of dimension $2$ and rank $3$. In canonical coordinates, it is given by
\[  \Phi(t)= \begin{pmatrix} \theta+\tau^2 & \tau^3 \\  1+\tau & \theta+\tau^2
\end{pmatrix}. \]
In \cite[Examples 6.4 \& 8.2]{maurischat2}, the second author showed furthermore, that a $\Fq(\theta)[t]$-basis $(e_1,e_2,e_3)$ of the motive $M(E)$ is given by $e_1=\tau\kappa_2$, $e_2=\kappa_2$, $e_3=\kappa_1$, and a $\Fq(\theta)[t]$-basis $(\de_1,\de_2,\de_3)$ of the comotive $N(E)$ is given by $\de_1=\dk_1\tau$, $\de_2=\dk_1$, $\de_3=\dk_2$.
Actually, the computations in \cite{maurischat2} didn't use any property of the base $\Fq(\theta)$, so they are valid over any $\Fq[t]$-algebra $R$, if we choose $\theta$ as the image of $t$ in $R$.

A straight forward, but tedious computation along the lines of the previous examples leads to 
\begin{align*}
 \Restau : \tau^*M(E)\otimes_{R[t]} N(E) &\to  R^\perf[t]dt, \\
 (\sum_{i} a_ie_i)\otimes (\sum_{j} b_j\de_j) & \mapsto  \begin{pmatrix} a_1^{(1)}, & a_2^{(1)},& a_3^{(1)}\end{pmatrix}\cdot 
 \begin{pmatrix}
 	1+g & 1 & -g  \\
 	1 & 0 & -1 \\
 	-g & -1 & g
 \end{pmatrix} 
 \cdot \begin{pmatrix} b_1\\ b_2\\ b_3 \end{pmatrix} dt
\end{align*}
where $g=\theta^q+\theta - 2t\in \mathbb{F}(\theta)[t]$.

Surprisingly, the image is even in $\mathbb{F}(\theta)[t]dt$ so no perfection is needed here.

\end{document}